\documentclass[11pt]{article}

\usepackage[margin=1in]{geometry}
\usepackage{amsmath,amsthm,amssymb,amsfonts,float}
\usepackage{listings}

\usepackage[dvipsnames]{xcolor}

\usepackage{cite}

\makeatletter
\newcommand{\subjclass}[2][2010]{%
  \let\@oldtitle\@title%
  \gdef\@title{\@oldtitle\footnotetext{#1 \emph{Mathematics subject classification.} #2}}%
}
\makeatother

\usepackage[hidelinks]{hyperref}
\usepackage{color}
\usepackage{esint}

\usepackage{mathtools,enumitem,mathrsfs}

\usepackage[compact]{titlesec}

\usepackage{comment}

\usepackage{todonotes}

\colorlet{darkblue}{blue!90!black}
\colorlet{darkred}{red!90!black}
\colorlet{dr}{red!90!black}

\mathtoolsset{showonlyrefs=true}
    
\newcommand{\ep}{\epsilon}
\newcommand{\eps}{\epsilon}

\newcommand{\leqc}{\lesssim}

\newcommand{\grad}{\nabla}

\newcommand{\norm}[1]{\left|\left| #1 \right|\right|}
\newcommand{\abs}[1]{\left| #1 \right|}
\newcommand{\set}[1]{\left\{ #1 \right\}}
\newcommand{\brak}[1]{\left\langle #1 \right\rangle} 

\newcommand{\R}{\mathbb{R}}
\newcommand{\Real}{\mathbb{R}}
\newcommand{\N}{\mathbb{N}}

\newcommand{\C}{\mathbb{C}}

\newcommand{\Z}{\mathbb{Z}}
\newcommand{\T}{\mathbb{T}}

\renewcommand{\S}{\mathbb{S}}

\newcommand{\Hbf}{{\bf H}}
\newcommand{\Wbf}{{\bf W}}
\newcommand{\Lbf}{{\bf L}}

\newcommand{\dee}{\mathrm{d}}
\newcommand{\ds}{\dee s}

\newcommand{\dv}{\dee v}

\newcommand{\dx}{\dee x}

\newcommand{\dy}{\dee y}

\DeclareMathOperator{\Div}{\mathrm{div}}
\DeclareMathOperator{\Id}{\mathrm{Id}}

\DeclareMathOperator{\curl}{\mathrm{curl}}

\usepackage{bbm,bm}

\renewcommand{\P}{\mathbf{P}}

\newcommand{\E}{\mathbf{E}}
\newcommand{\EE}{\mathbf E}
\newcommand{\PP}{\mathbf P}

\newtheorem{theorem}{Theorem}[section]
\newtheorem{proposition}[theorem]{Proposition}
\newtheorem{corollary}[theorem]{Corollary}
\newtheorem{lemma}[theorem]{Lemma}
\newtheorem*{lemma*}{Lemma}
\newtheorem{claim}[theorem]{Claim}

\newtheorem{assumption}{Assumption}

\theoremstyle{definition}
\newtheorem{definition}[theorem]{Definition}
\newtheorem{remark}[theorem]{Remark}

\numberwithin{equation}{section}
    
\begin{document}
\title{Negative regularity mixing for random volume preserving diffeomorphisms}
\author{Jacob Bedrossian\thanks{\footnotesize Department of Mathematics, University of California, Los Angeles, CA 90095, USA \href{mailto:jacob@math.ucla.edu}{\texttt{jacob@math.ucla.edu}}. JB was supported by NSF Award DMS-2108633.} \and Patrick Flynn\thanks{\footnotesize Department of Mathematics, University of California, Los Angeles, CA 90095, USA \href{mailto:pflynn@math.ucla.edu}{\texttt{pflynn@math.ucla.edu}}. } \and Sam Punshon-Smith \thanks{ \footnotesize Department of Mathematics, Tulane University, New Orleans, LA 70118, USA \href{mailto:spunshonsmith@tulane.edu}{\texttt{spunshonsmith@tulane.edu}}. This material was based upon work supported by the NSF Award DMS-1803481.} } 
\maketitle

\begin{abstract}
We consider the negative regularity mixing properties of random volume preserving diffeomorphisms on a compact manifold without boundary.
We give general criteria so that the associated random transfer operator mixes $H^{-\delta}$ observables exponentially fast in $H^{-\delta}$ (with a deterministic rate), a property that is false in the deterministic setting.
The criteria apply to a wide variety of random diffeomorphisms, such as discrete-time iid random diffeomorphisms, the solution maps of suitable classes of stochastic differential equations, and to the case of advection-diffusion by solutions of the stochastic incompressible Navier-Stokes equations on $\T^2$. In the latter case, we show that the zero diffusivity passive scalar with a stochastic source possesses a unique stationary measure describing ``ideal'' scalar turbulence.
The proof is based on techniques inspired by the use of pseudodifferential operators and anisotropic Sobolev spaces in the deterministic setting.
\end{abstract}

\section{Introduction}
\label{sec:intro}

In this paper we study negative regularity mixing by volume-preserving random diffeomorphisms $\{\phi^t\,:\, t\in \mathbb{I}\}$ on a smooth compact Riemannian manifold $\mathcal{M}$ without boundary. Here, the index set $\mathbb{I}$ is either $\R_+$ or $\Z_+$. One of our main motivations comes from fluid mechanics, namely the advection equation
\begin{align}
  \partial_t f_t + u_t \cdot \nabla f_t = 0,\quad f|_{t=0} = f_0  \label{eq:PassScalIntro}
\end{align}
where $f_t$ is a scalar field and $(u_t)_{t \in \R_+}$ is a Lipschitz regular, divergence-free, (time-dependent) velocity field with $\mathcal{M} = \mathbb T^d$.
Naturally, the solution is given by $f_t = f_0 \circ (\phi^{t})^{-1}$ where $\phi^t = \phi^t_{u_0,\omega}$ is the flow map of associated to the velocity field $u_t$. The solution operator
\[
T^t: f_0 \mapsto f_t = f_0 \circ (\phi^t)^{-1}
\]
is often called the transfer operator in the dynamical systems literature.

In the context of Anosov maps, the study of mixing properties of the map $\phi^t$ is closely related to the spectrum of the transfer operator, known as Ruelle Resonances. Introduced by Ruelle within the framework of thermodynamic formalism~\cite{Ruelle-Thermodynamic-2004t,Ruelle1986-jf,Ruelle-Locating-1986c}, these resonances offer refined insights into the decay of correlations. Subsequent work has significantly advanced our understanding of these resonances through the use of anisotropic Sobolev spaces~\cite{Blank2002-vi,Gouezel2005-ff,Liverani2005-qg,Baladi2005-ro,Baladi2007-ey,demers2018gentle,baladi2017quest} and microlocal analysis~\cite{faure2008semi,Faure-Prequantum-2018h,Faure2011-mn,Faure-Semiclassical-2014i,Faure-Semiclassical-2017e,Faure-Fractal-2023u,Faure-Micro-local-2021v}. See section \ref{sec:comments} for more context on some of this work in the context of our main result.

In the case when the velocity field is random (see \cite{bedrossian2021almost,bedrossian2022almost,cooperman2024exponential,dolgopyat2004sample}), it is possible to prove a \emph{quenched mixing result} (or equivalently {\em almost-sure exponential decay of correlations}), namely there exists a deterministic $\gamma > 0$ and a random constant $D = D(\omega)$ satisfying $\EE D^2 <\infty$, such that for all mean zero $f_0 \in H^1$ there holds
\begin{align}
\norm{f_t}_{H^{-1}} \leq D(\omega)e^{-\gamma t} \norm{f_0}_{H^1}. \label{ineq:quenched}
\end{align}
In the case of \eqref{eq:PassScalIntro} or similar examples, the random constant $D$ will also depend on the initial velocity field $u_0$ (but not the initial data $f_0$).

Due to the time-reversibility of transport equations, it is relatively easy to conclude that one cannot obtain an almost-sure exponential decay result like \eqref{ineq:quenched} $H^{-1}$ decay without assuming $f_0$ having some small amount of regularity say $H^{\ep}$. 
In particular one cannot expect decay in $H^{-1}$ if $f_0$ is also taken in $H^{-1}$.

The goal of this work is to prove that for many random maps, negative regularity mixing does hold in an \emph{averaged} sense, namely that for all mean-zero $f_0 \in H^{-\delta}$, there exists a $\mu > 0$ such that
\begin{align*}
\EE \norm{f_t}_{H^{-\delta}}^{2} \lesssim e^{-\mu t} \norm{f_0}_{H^{-\delta}}^{2}. 
\end{align*}
for some fixed $\mu,\delta > 0$ (independent of $f_0$), and also that it holds in a quenched sense but with a constant that depends on $f_0$ (Corollary \ref{cor:quenched}).  

After we state the main results, we include a more extended discussion of both the results and the motivations, however, one can liken the result to something like the non-random multiplicative ergodic theorem \cite{kifer2012ergodic}. Indeed, for each individual scalar field $f_0 \in H^{-\delta}$ (and initial velocity field $u_0$ if applicable) there may exist very specific realizations of $(u_t)$ which obtain arbitrarily bad decay rates in $H^{-\delta}$, but these be so rare that if one averages over the ensemble of velocity fields, they are essentially negligible. 
The motivation and proof draws significant inspiration from the work on anisotropic Sobolev spaces, especially \cite{faure2008semi}.

\subsection{Abstract framework and statement of negative regularity mixing}

Here we outline a general abstract setting for our theorem that applies broadly to a whole class of random maps, including iid random diffeomorphisms, stochastic transport on compact manifolds (i.e. particle trajectories solve SDEs) \cite{Gess2021-wh,Coti-Zelati2024-az}, advection by velocity fields generated by the stochastic Navier-Stokes equations, alternating random shear flows \cite{blumenthal2023exponential}, and other similar settings (see Section \ref{sec:Concrete}).
Our abstract framework is similar to that of \cite{blumenthal2023exponential}, but more general in order to treat the cases studied in \cite{bedrossian2022lagrangian,bedrossian2022almost,bedrossian2021almost,bedrossian2022batchelor}. See Section \ref{sec:Concrete} for explanation on how to connect our abstract result to some of the concrete examples. 

Let $M$ be a $d$-dimensional, smooth, compact Riemannian manifold without boundary and $(\Omega,\mathscr{F},\P)$ be a probability space, $\theta:\Omega\to\Omega$ be a $\P$ preserving transformation and $(u_n)$ be a Feller Markov chain on a Polish space $U$ with a {\em unique} stationary (probability) measure $\mu$ and an associated continuous random dynamical system $\Phi^n_{\omega}:U\to U$. We denote the corresponding skew product flow on $Z = \Omega\times U$ by $\tau(\omega,u) := (\theta\omega,\Phi_{\omega}(u))$ with invariant ergodic product measure $\mathfrak{m} = \P\times \mu$. 

We will consider $\phi_{\omega,u}: M \to M$ a $C^{k_0}$-regular (with $k_0 \geq 2$), random volume preserving diffeomorphism and denote
\[
  \phi^n_{\omega,u} = \phi_{\tau^n(\omega,u)}\circ\ldots\circ \phi_{\tau(\omega,u)}\circ\phi_{\omega,u}
\]
the $n$-fold composition along the skew product flow $\tau$.
We emphasize that we do not assume that each map $\phi_{\omega,u}$ in this composition is iid, so the associated process $x_n = \phi^n_{\omega,u} (x)$, $x\in M$ is \emph{not} assumed Markovian in general.
However, the joint process $(x_n,u_n) = (\phi^n_{\omega,u}(x),\Phi^n_\omega(u))$ on $M\times Z$ is Markovian.

\begin{remark}
To include the case of iid random diffeomorphisms in our framework, one simply dispenses with $(u_n)$ and the relevant Markov process becomes $x_n = \phi^n_\omega(x)$ where
\begin{align*}
\phi^n_{\omega} = \phi_{\theta^n\omega}\circ\ldots\circ \phi_{\theta \omega}\circ\phi_{\omega}. 
\end{align*}
This includes the time-1 maps of well-posed, volume-preserving SDEs; see \cite{Kunita1997-oa}. 
\end{remark}

\subsubsection{Lyapunov structure}
First, we make some assumptions which provide control on how large the derivatives of $\phi^n_{\omega,u}$ can be.
In most settings, $(u_n)$ is the velocity field, and the deviations on $\phi^n$ are stated in terms of the deviations of $(u_n)$.  
Recall that a function $V:U \to [1,\infty)$ is called a \emph{Lyapunov function} for $(u_n)$ if $V$ has bounded sublevel sets $\{V \leq r\}$ and the associated Markov kernel for $(u_n)$, $P(u,A) = \P_{u}(u_1\in A)$ satisfies a Lyapunov-Foster Drift condition: there exists $\delta \in (0,1)$ and $K \geq 0$ such that
\begin{equation}\label{eq:Lyapunov-Foster}
  PV \leq \delta V + K,
\end{equation}
where $PV(u) := \int_{U} V(u^\prime) P(u,\dee u^\prime)$.

In general, we will require a stronger Lyapunov structure on $(u_n)$ and the derivatives of $\phi_{\omega,u}$.
Specifically, we will assume the existence of a two parameter family of Lyapunov functions $V_{\beta,\eta} : U \to [1,\infty)$ defined for $\beta \geq 0$ and $\eta \in (0,1)$, satisfying the following conditions: 
\begin{assumption}[Lyapunov Structure]
\label{ass:Lyapunov-Flow-Assumption}
There exists a two parameter family of Lyapunov functions 
\[
  \{V_{\beta,\eta}\,:\, \beta \geq 1, \eta \in (0,1)\},
\]
satisfying for each $a \in (0,1/\eta)$, $V_{\beta,\eta}^a(u) = V_{\beta a, \eta a}(u)$, such that the following condition holds: $\exists a_* > 1$ such that $\forall \beta \geq 1$, $\forall k \in \N_{\geq 1}$ with $k \leq k_0$ and $\forall b\geq 0$ there exists a $\beta(k,b)\geq \beta$ satisfying $\beta(1,b) = \beta$ such that $\forall \eta \in (0,1/a_*)$, 
\begin{equation}\label{eq:Ck-bound-assumption}
\E_{u}\left(\mathcal Q_k(\phi)^{b} V_{\beta,\eta}(u_1) \right)^{a_*}\leqc_{b,k,a_*,\beta,\eta} V_{\beta(k,b),\eta}(u),
\end{equation}
where $\mathcal Q_k(\phi)$ is given in Definition \ref{def:manifoldCkNorm}. 
\end{assumption}

In what follows, we will omit the parameters $\beta,\eta$ in contexts where they are not important. 

\begin{remark}

Since we are working on a Riemannian manifold, the definition of $\mathcal{Q}_k(\phi)$ is somewhat technical and is defined over a suitable atlas. However, intuitively $\mathcal Q_k(\phi)$ is a measure of the size of the $k$th derivative of $\phi$ and $\phi^{-1}$, heuristically
\[
\mathcal{Q}_k(\phi) \,``="\, \sup_{|\alpha|\leq k }\|D^\alpha\phi\|_{L^\infty} + \|D^\alpha\phi^{-1}\|_{L^\infty}.
\]

The assumption on $\mathcal Q_k(\phi)$ in \eqref{eq:Ck-bound-assumption} is necessary for our pseudodifferential operator method, and is particularly important in bounding the errors produced in Egorov's theorem (see Appendix \ref{appendix:PDOs} below). 
See Section \ref{sec:Concrete} for discussions on the possible forms of $V_{\beta,\eta}$ and proofs of \eqref{eq:Ck-bound-assumption}. 
\end{remark}

\begin{remark}\label{eq:super-lyapunov-remark}
Note, that upon setting $b =0$ in \eqref{eq:Ck-bound-assumption}, an application of Jensen's inequality implies that $V = V_{\beta,\eta}$ satisfies the {\em super-Lyapunov property}, namely that for {\em every} $\ep \in (0,1)$ there is a $K_\ep>0$ such that
\[
  PV \leq (CV)^{1/a_*}\leq \ep V + K_{\ep}.
\]
This is a much stronger condition than a standard Lyapunov-Foster condition \eqref{eq:Lyapunov-Foster}. Nevertheless, it is satisfied for many semilinear, parabolic stochastic PDEs where the nonlinearity is sub-critical with respect to a conserved quantity (e.g. the 2d Stochastic Navier-Stokes equations).
\end{remark}

\subsubsection{The linearization and projective process} \label{sec:LinProj}

Next, we will need some assumptions that encode dynamical information. 
The first is the assumption on the spectral properties of the `projective' Markov semigroup, described below. 
A prominent role here is played by the inverse transpose of the derivative of the flow map, which we denote by $\check{A}_{\omega,u,x}$, defined by
\[
  \check{A}_{\omega,u,x} := (D_x\phi_{\omega,u})^{-\top},
\]
where the transpose is taken with respect to the Riemannian metric on $M$. At each $x\in M$, $\check{A}_{\omega,u,x}$ naturally acts on covectors $\xi \in T^*_xM$ with $\check{A}_{\omega,u,x}\xi \in T^*_{\phi_{\omega,u}(x)}M$ and $\check{A}_{\omega,u,x}^n = (D_x\phi^n_{\omega,u})^{-\top}$ is a linear cocycle on $T^*M$ since
\[
  \check{A}^n_{\omega,u,x} = (D_x\phi^n_{\omega,u})^{-\top} = \check{A}_{\tau^n(\omega,u),\phi^n_{\omega,u}(x)}\circ\ldots\circ \check{A}_{\omega,u,x}.
\]

\begin{remark}
Under Assumption \ref{ass:Lyapunov-Flow-Assumption}, by the multiplicative ergodic theorem \cite{kifer2012ergodic,raghunathan1979proof, oseledets1968multiplicative,viana2014lectures},  the linear cocycle $\check{A}^n_{\omega,u,x}$ has a Lyapunov spectrum $\check{\lambda_1} \geq \check{\lambda_2}\geq \ldots \check{\lambda_d}$ which have the property that $\sum_{j}\check{\lambda}_j  =0$ since $\mathrm{det}(\check{A}_{\omega,u,x}) = 1$. This can be related to the Lyapunov spectrum of the linear cocycle $A^n_{\omega,u,x} = D_x\phi^n_{\omega,u}$ on $TM$ $\lambda_1 \geq \lambda_2 \geq \ldots \geq \lambda_d$ via
\[
  \check{\lambda}_j = -\lambda_{d-j+1},
\]
(see e.g. \cite{blumenthal2023exponential} Lemma B1).
Specifically, if $\lambda_1>0$, by volume preservation, it must be that $\check{\lambda}_1>0$.
\end{remark}

In what follows, we consider $(x_n,\xi_n)$, with $\xi_n = \check{A}^n_{\omega,u,x}\xi$, the induced dynamics on the cotangent bundle $T^*M$. Let $T^*\!P$ denote the Markov semigroup associated to the Markov chain $(u_n,x_n,\xi_n)$ on $U\times T^*M$ defined for each bounded measurable $\varphi$ by
\[
  T^*\!P\varphi(u_0,x_0,\xi_0) := \E\varphi(u_1,x_1,\xi_1).
\]
Additionally denote $(x_n,v_n)$, with
\[
  v_n = \check{A}^n_{\omega,u,x}v/|\check{A}^n_{\omega,u,x}v|,
\]
the ``projectivized'' dynamics on the unit cosphere bundle $\S^*M = \{v\in T^*M\,:\, |v|_* = 1\}$ and denote $\hat{P}$ the Markov semigroup associated to the Markov chain $(u_n,x_n,v_n)$ on $U\times \S^*M$ defined analogously.

Given a Lyapunov function $V(u)$ for $(u_n)$, we will denote the space $\mathrm{Lip}_V(U\times \S^*M)$ to be the space of continuous functions $\psi(u,x,v)$ on $U\times \S^*M$ with finite weighted Lipschitz norm
\begin{equation}\label{eq:lipnorm}
  \|\psi\|_{\mathrm{Lip}_V} := \sup_{z\in U\times S^*M}\frac{|\psi(z)|}{V(u)} + \sup_{z,z^\prime\in U\times S^*M}\frac{\|\psi(z) - \psi(z)\|}{(V(u) + V(u^\prime))d(z,z^\prime)} <\infty, 
\end{equation}
where for $z=(u,x,v)$ and $z^\prime= (u^\prime,x^\prime,v^\prime)$ $d(z,z^\prime) = d(u,u^\prime) + d_{S^*M}\left((x,v),(x^\prime,v^\prime)\right)$,
with $d$ being the metric on $U$ and $d_{S^*M}$ the metric on $S^*M$.
It follows by the Feller property of $(u_n)$ and assumption \eqref{eq:Ck-bound-assumption} that $\hat{P}$ restricts to a bounded linear operator on the space $\mathrm{Lip}_V(U\times S^*M)$, at least if one replaces the time-step $1$ with $N_0$ and considers the process $(\tilde{u}_n,\tilde{v}_n) = (u_{n N_0 }, v_{n N_0})$ (see [\cite{bedrossian2022almost} Lemma 5.2] for more information).
By relabeling, we can assume without loss of generality that $N_0 = 1$. 

Of particular interest are functions which are $-p$ homogeneous, i.e. 
\[
\varphi(u,x,\xi) = |\xi|^{-p}\psi(u,x,\xi/|\xi|),
\]
where $\psi(u,\cdot)$ is a smooth function on $S^*M$ which is homogeneous of degree $0$ in $\xi$.
Evaluating $T^*P$ on such a symbol, we obtain
\[
T^*P\varphi(u,x,\xi) = \hat{P}^p\psi(u,x,\xi/|\xi|),
\]
where $\hat{P}^p$ is the `twisted' semigroup defined by
\[
\hat{P}^p\psi(u,x,v) = \E_{u,x,v} |\check{A}_{\omega,u,x}v|^{-p}\psi(u_1,\phi^1_{\omega,u}(x),v_1).
\]
Similar to $\hat{P}$, $\hat{P}^p$ forms a semigroup of bounded operators on $\mathrm{Lip}(U\times S^*M)$ (again, possibly by increasing the time-step and relabeling) and that when $p=0$ one recovers the projective semigroup $\hat{P}$.
We make the following spectral assumption on $\hat{P}^p$. 
\begin{assumption}[Spectral Gap]\label{ass:SpecGap} \
\begin{enumerate} 
\item[(i)] There exists a $p_0 > 0$ such that for all $p\in [-p_0,p_0]$, $\hat{P}^p$ admits a simple dominant eigenvalue $e^{-\Lambda(p)}$ on $\mathrm{Lip}_V(U\times S^*M)$ with a spectral gap, namely there exists an $r \in (0,e^{-\Lambda(p)})$ such that the spectrum of $\hat{P}^p$ satisfies
  \[
    \sigma(\hat{P}^p)\backslash\{e^{-\Lambda(p)}\} \subseteq B(0,r),
  \]
where $B(0,r)\subset \C$ denotes the open ball of radius $r$ centered at $0$.
\item[(ii)] The eigenvalue satisfies $\Lambda(p) > 0$ for $p \in (0,p_0)$. 
  \item[(iii)] Let $\pi_p$ be the rank-one spectral projector associated with $e^{-\Lambda(p)}$. Then $\psi_p \equiv  \pi_p\bm{1}$ is a dominant eigenfunction satisfying
  \[
  \hat{P}^p\psi_p = e^{-\Lambda(p)}\psi_p,
  \]
  and is bounded below by a positive constant on bounded sets. That is for each bounded set $B \subset U$ there is a $c_B >0$ such that $\inf_{u\in B}\psi_p > c_B$.
\end{enumerate}
\end{assumption}   

\begin{remark}
Assumption (iii) is essentially the assumption that $(u_n,x_n,v_n)$ is irreducible; see \cite{blumenthal2023exponential,bedrossian2022almost} for discussion.  
\end{remark}

\begin{remark}
Irreducibility and geometric ergodicity of the projective process $(u_n,x_n,v_n)$ in a suitable Wasserstein metric, usually proved via a weak Harris' theorem (see \cite{HM06}), 
implies Part (i) and (iii) with a straightforward spectral perturbation argument (see e.g. \cite{blumenthal2023exponential,bedrossian2021almost} and the references therein).

Note that also, in particular, by setting $p = 0$ Assumption \ref{ass:SpecGap} also {\em implies} geometric ergodicity (again in a suitable metric) of $\hat{P}$ and hence a unique stationary probability measure $\nu \in \mathcal{P}(U\times S^*M)$.
\end{remark}

\begin{remark} 
The quantity $\Lambda(p)$ is called the \emph{moment Lyapunov function} and, at least for $\abs{p}$ close to $0$, Part (iii) implies that for $\PP \times \nu$-almost every $(\omega,u,x,v)$, 
\begin{align*}
\Lambda(p) = -\lim_{n \to \infty} \frac{1}{n} \log \EE \abs{\check{A}^n_{\omega,u,x}v}^{-p}.
\end{align*}
The function $\Lambda(p)$ is also closely related to the probability of finding $\abs{\check{A}^n_{\omega,u,x}v}$ far from the expected $e^{\check{\lambda}_1 n}$ for large $n$. Particularly, it is related to a large deviation principle for the projective chain $(u_n,x_n,v_n)$. See e.g. \cite{baxendale1988large,baxendale1993kinematic} for more discussions. 

One can readily prove that $\Lambda'(0) = \check{\lambda}_1$ (see [Lemma 5.10 \cite{bedrossian2022almost}] for the proof when the cocyle is $A^n_{\omega,u,x}$ instead of $\check{A}^n_{\omega,u,x}$ ).
Hence, positivity of the Lyapunov exponent $\check{\lambda}_1>0$ implies that for $p \ll 1$
\[
\Lambda(p) = p\check{\lambda}_1 + o(p) >0, 
\]
which would then imply Part (ii) of Assumption \ref{ass:SpecGap}. 
\end{remark}

Given the connection between $T^*P$ and $\hat{P}^p$, Assumption \ref{ass:SpecGap} implies that the function
\[
  a_p(u,x,\xi) := |\xi|^{-p}\psi_p(u,x,\xi/|\xi|)
\]
is an eigenfunction of $T^*\!P$ with eigenvalue $e^{-\Lambda(p)}$
\[
  T^*\!P a_p = e^{-\Lambda(p)}a_p.
\]
The function $a_p$ is the starting point of the symbol of a pseudo-differential operator that plays a key role in our proof below.

\subsubsection{Two point geometric ergodicity}

Closely related to the behavior of the twisted Markov semi-group $\hat{P}^p$ and the associated eigenfunctions $\psi_p$, is the behavior of the two point motion. Namely, let $x_n$ and $y_n$ be two different trajectories on $M\times M$ starting from distinct points $x_0$ and $y_0$, that is $x_n = \phi^n_{\omega,u}(x_0)$ and $y_n = \phi^n_{\omega,u}(y_0)$.
In order to avoid reducing to the one point motion, we assume that the starting points $(x_0,y_0)$ do not belong to the diagonal set
\[
  \Delta = \{(x,x)\,:\, x\in M\} \subset M\times M. 
  \]
We denote the associated Markov chain $(u_n,x_n,y_n)$ on $U\times \Delta^c$ by $P^{(2)}$. We remark that $U \times \Delta^c$ is an almost-surely invariant set for the two-point semigroup defined over $U \times M \times M$.
As the space $\Delta^c$ is not compact, one is likely to use a Lyapunov function $\mathcal{V}$ on $U\times \Delta^c$ to control the behavior of the two point motion. As in \cite{bedrossian2022almost,bedrossian2021almost,blumenthal2023exponential} we consider the Lyapunov function $\mathcal{V}_p$ with the following properties:
\begin{align}
\mathcal{V}_p(u,x,y) \approx (d(x,y)^{-p}\vee 1) V_{\beta,\eta}(u) \label{ineq:Vp}
\end{align}
for all $(u,x,y) \in U \times \Delta^c$.
Depending on the setting, this kind of Lyapunov function can be constructed using $\psi_p$ \cite{bedrossian2022almost,bedrossian2021almost,blumenthal2023exponential} or using large deviation estimates on the exit times \cite{dolgopyat2004sample,baxendale1988large}. 
We will need the following assumption on the exponential decay of the two point motion.
\begin{assumption}[Exponential Decay of Two Point Motion]\label{ass:2pt}
  We assume that for all $0 < p \ll 1$ sufficiently small, for all $\beta \geq 1$, and all $\eta \in (0,1)$, $\exists \mathcal{V}_p$ satisfying \eqref{ineq:Vp} which is a Lyapunov function for the two point motion $(u_n,x_n,y_n)$ and for which the Markov process is $\mathcal{V}_p$-geometrically ergodic, namely there exists a $\gamma > 0$ such that for all $x_0,y_0\in M$ and all $\varphi\in C_{\mathcal{V}_p}(U\times \Delta^c)$, there holds
\begin{align*}
\abs{P^{(2)} \varphi(u,x,y) - \int_U \int_M \int_M \varphi\, \dee \mu \dee x \dee y} \lesssim \mathcal{V}_p(u,x,y) \norm{\varphi}_{C_{\mathcal{V}_p}} e^{-\gamma t}. 
\end{align*}
\end{assumption}
Heuristically, we can consider Assumption \ref{ass:2pt} as a strict improvement over Assumption \ref{ass:SpecGap} above, extending the linearized information to the nonlinear two-point dynamics.
Indeed, the proofs of \cite{bedrossian2021almost,bedrossian2022almost,blumenthal2023exponential} used Assumption \ref{ass:SpecGap} together with a Harris' theorem argument using the Lyapunov function $\mathcal{V}_p$ to prove Assumption \ref{ass:2pt}.
However, it is easier to state these as separate assumptions, rather than list all of the individual assumptions required for the Harris' theorem argument to apply.

By a now-standard Borel-Cantelli argument together with a little additional Fourier analysis (see \cite{dolgopyat2004sample} and \cite{bedrossian2022almost,cooperman2024exponential,blumenthal2023exponential}), Assumption \ref{ass:2pt} implies quenched mixing estimates such as \eqref{ineq:quenched}.
We technically do not use such estimates directly, however, the manner in which we employ Assumption \ref{ass:2pt} is nevertheless quite similar to a quenched mixing estimate.

\subsubsection{Main result}\label{sec:main-result}

In what follows, denote $f_n = f_0 \circ (\phi^n_{\omega,u})^{-1}$.
Our main result is the following theorem. 
\begin{theorem} \label{thm:AM}
Under Assumptions \ref{ass:Lyapunov-Flow-Assumption}--\ref{ass:2pt}, there exists constants $\mu >0$, $\delta \in (0,1)$ and $\beta_* \geq 1$, such for each Lyapunov function $V = V_{\beta_*,\eta}$, with $\eta >0$ sufficiently small there holds the following: for all distributions $f_0 \in H^{-\delta}$ with $\int f_0 \dee x = 0$, for all $n \geq 0$, $u_0 \in U$
\begin{equation}
\label{eq:annealedeq}
  \E \left[V(u_n)^{-1}\|f_n\|_{H^{-\delta}}^2\right] \leqc V(u_0) e^{-2\mu n}\|f_0\|_{H^{-\delta}}^2,
\end{equation}
As a corollary, there holds $\forall q\in (0,2)$
\begin{align}
\E \|f_n\|_{H^{-\delta}}^{q}\leqc_q  V^{q}(u_0) e^{-q\mu n} \|f_0\|_{H^{-\delta}}^{q}. \label{ineq:annealedNoV}
\end{align}
\end{theorem}

\begin{remark}
We do not currently know if the result holds for $q > 2$, nor do we know how large one can make $\delta$; see Section \ref{sec:comments} for more discussion.
\end{remark}

\begin{remark}
In the case that $\phi^1_{\omega,u} = \phi^1_\omega$ are iid diffeomorphisms (as would happen with the stochastic flow of diffeomorphisms generated by SDEs with smooth vector fields \cite{Kunita1997-oa}) we can dispense with $V$ and hence \eqref{ineq:annealedNoV} holds for $\eps = 0$ and simply becomes
\begin{equation}
\label{eq:annealedeq}
  \E \|f_n\|_{H^{-\delta}}^2 \leqc e^{-2\mu n}\|f_0\|_{H^{-\delta}}^2.
\end{equation}
\end{remark}

\begin{remark}
For distributions that cannot be identified with an $L^1$ function, we can define $\int f_0 \dee x$ as the distribution acting on the smooth function $\phi \equiv 1$. 
\end{remark}

\begin{remark}
A related result was proven in \cite{Coti-Zelati2024-az}, in the setting of the Kraichnan model, where the authors prove an exact identity for the averaged exponential decay of negative Sobolev norms. However, the proof in \cite{Coti-Zelati2024-az} relies on the specific structure of the Kraichnan model and does not extend to the general setting considered here. Additionally, as the self similarity required in \cite{Coti-Zelati2024-az} makes it challenging to define an associated flow map, their results do no immediately follow from Theorem \ref{thm:AM}.
\end{remark}

As a corollary of Theorem \ref{thm:AM}, we obtain the following quenched mixing estimate with a random constant depending on the initial data.
\begin{corollary}[Quenched Result]\label{cor:quenched}
Let $\mu$ and $\delta$ and $\beta_*$ be as in Theorem \ref{thm:AM}. Then, for all $f_0 \in H^{-\delta}$ with $\int f_0 \dee x = 0$, and $V = V_{\beta_*, \eta}$ with $\eta$ suitably small, there exists a random constant $K(f_0,u_0)$ (depending on $u_0$ and the initial $f_0$) such that there holds for all $n \geq 0$ and $u_0 \in U$,
\[
\|f_n\|_{H^{-\delta}} \leqc K(u_0,f_0) e^{-\mu n/2} V(u_0)\|f_0\|_{H^{-\delta}}.
\]
Moreover, $K$ has uniform (in $f_0$ and $u_0$) moments $\E K^q \leqc_q 1$, $\forall q\in (0,2)$.
\end{corollary}
\begin{proof}
Clearly we can write
\[
  \|f_n\|_{H^{-\delta}} \leqc K(u_0,f_0) e^{-\mu n/2} V(u_0)\|f_0\|_{H^{-\delta}}
\]
where $K$ is defined by
\[
K(u_0,f_0) = \max_{n\geq 0}\frac{e^{\mu n/2}\|f_n\|_{H^{-\delta}}}{V(u_0)\|f_0\|_{H^{-\delta}}}.
\]
To see that $K$ has finite moments we estimate
\[
\E K^q \leq \sum_{n\geq 0}\frac{e^{\mu n q/2}\E\|f_n\|_{H^{-\delta}}^q}{V(u_0)^q\|f_0\|_{H^{-\delta}}^q} \leqc_q \sum_{n\geq 0} e^{-q\mu n/2} \leqc_q 1.
\]
\end{proof}

\subsection{Comments on the proof of Theorem \ref{thm:AM}} \label{sec:comments}

The proof of Theorem \ref{thm:AM} is inspired primarily by the pioneering work of Faure, Roy and Sjostrand \cite{faure2008semi} and the subsequent works \cite{Faure-Prequantum-2018h,Faure2011-mn,Faure-Semiclassical-2014i,Faure-Semiclassical-2017e,Faure-Fractal-2023u,Faure-Micro-local-2021v}. In these works the authors use microlocal analysis to construct special anisotropic Sobolev spaces (i.e. spaces with different regularity in different directions in frequency space) on which one can prove spectral a gap, or at least quasi-compactness, for the transfer operator. This microlocal approach as also been used in the context of Anosov flows generated by a vectorfield \cite{Dyatlov2016-ss,Dyatlov2015-uj}. See also \cite{Blank2002-vi,Gouezel2005-ff,Liverani2005-qg,Baladi2005-ro,Baladi2007-ey,demers2018gentle,baladi2017quest} for other works using anisotropic Banach spaces of functions to study spectra of the transfer operator.

The motivation of these anisotropic spaces is to find a space $X$ with
\begin{align*}
H^{-\alpha} \subset X \subset H^{\alpha}, 
\end{align*}
such that for average-zero observables, the scalars decay exponentially in $X$ due to a spectral gap of the transfer operator $\mathcal{T}f := f \circ \phi$ (or to at least prove a quasi-compactness estimate that implies localization of the essential spectrum) 
\begin{align}
\norm{f \circ \phi^n}_X \lesssim e^{-\mu n} \norm{f}_X. \label{ineq:Xspec}
\end{align}
In settings where this is possible, it yields a significantly more precise result than quenched mixing estimates.  
In the microlocal approach of \cite{faure2008semi} (also in \cite{Dyatlov2015-uj,Dyatlov2016-ss}), these spaces are found by building a Lyapunov function for the linearized process on $T^\ast M$, $a(x,\xi)$, and using it to define a pseudo-differential operator with a suitable quantization procedure.
The space one obtains the essential spectrum estimates in is then defined via the norm (at least heuristically)
\begin{align*}
\norm{f}_X = \|\mathrm{Op}(a)f\|_{L^2},
\end{align*}
where $\mathrm{Op}(a)$ denotes the pseudo-differential operator associated to the symbol $a$. 
In the case of Anosov diffeomorphisms, one chooses $a(x,\xi) \approx \brak{\xi}^{m}$ for $\xi$ on the tangent to the unstable manifold through $x$ and $\approx \brak{\xi}^{-m}$ for $\xi$ on the tangent to the stable manifold through $x$ (the angle between the tangents is uniformly bounded away from zero by assumption of uniform hyperbolicity). 
After a suitable regularization procedure and a variable-order Egorov's theorem, one can use this symbol to construct a norm $X$ and prove a Lasota-Yorke-type estimate which implies quasi-compactness of the transfer operator (see \cite{faure2008semi} for more details).

Again considering the case of Anosov diffeomorphisms, one can observe that a negative regularity mixing estimate such as Theorem \ref{thm:AM} cannot possibly hold for deterministic maps. Indeed, by concentrating the initial distribution in $H^{-\delta}$ close to the stable manifolds, one can obtain arbitrarily slow mixing rates for certain sequence of pathological initial data. 
For the random map case, the difference here is that for each fixed initial distribution, the assumptions in Section \ref{sec:LinProj} basically rule out the possibility that a positive probability set of $\omega \in \Omega$ results in stable manifolds that line up badly to produce poor decay rates. 
One can wonder exactly how large of a moment one can take in Theorem \ref{thm:AM}, i.e. under what conditions the set of bad $\omega$ remains negligible; our proof currently yields at most second moments. 
Similarly, one can wonder how large $\delta$ can be taken, i.e. whether one can take $\delta$ all the way down to $\delta > (d-1)/2$ or even $\delta > d/2$ (going further is clearly impossible as point masses cannot be mixed). 

In order to prove Theorem \ref{thm:AM}, we use a Lyapunov function (now again in the stochastic process sense) for the linearized process on $U \times T^\ast M$ defined in Section \ref{sec:LinProj} and use this as a symbol $a$ to obtain exponential decay estimates on the quantity
\begin{align*}
\brak{\mathrm{Op}(a)f,f}. 
\end{align*}
We use a regularization of the symbol
\begin{align*}
a(u,x,\xi) = \frac{1}{\abs{\xi}^p}\psi_p(u,x,\xi), 
\end{align*}
which by Assumption \ref{ass:SpecGap}, satisfies a spectral-gap type estimate for the linearized process Markov semigroup. 
Formally, we could then expect by Egorov's theorem
\begin{align*}
\EE \brak{\mathrm{Op}(a)f \circ \phi^n ,(f\circ\phi^n)} \approx e^{-\Lambda(p) n} \brak{f,\mathrm{Op}(a) f}.
\end{align*}
Given the lower bounds on $\psi_p$ in Part (iii) of Assumption \ref{ass:SpecGap}, by G{\aa}rding's inequality, we can heuristically expect that the quantization of this symbol would satisfy something like (ignoring $V$ for a moment) 
\begin{align*}
\brak{\mathrm{Op}(a) f,f} \lesssim \norm{f}_{H^{-p/2}}^2,
\end{align*}
which would suggest Theorem \ref{thm:AM}. 
The most obvious way this intuition fails is that both Egorov's theorem and G{\aa}rding's inequality only control high frequencies and both leave an error in lower regularity, for example the most one can hope for from G{\aa}rding's inequality is something like the following for some constants $\zeta,C> 0$, 
\begin{align}
\zeta \norm{f}_{H^{-p/2}}^2 - C \norm{f}_{H^{-\frac{d+3}{2}}}^2 \leq \brak{f,\mathrm{Op}(a) f}. \label{ineq:GardingIntro}
\end{align}
This is analogous to why one only directly obtains Lasota-Yorke type estimates in the deterministic setting, rather than direct exponential decay estimates.
These errors will need to be dealt with carefully and in particular, require us to first prove a $H^{-p/2} \to H^{-\frac{d+3}{2}}$-type mixing estimate (this is done using Assumption \ref{ass:2pt}; see Section \ref{sec:lowDecay}).
Note that even this kind of estimate is false for deterministic maps. 
However, there are two more reasons the above intuition is na\"{i}ve: (A) the regularity of $\psi_p$ is limited to $C^1$ and so the symbol must be suitably regularized and, perhaps most importantly, (B) the positivity and regularity of the symbol depend badly on the derivatives of $\phi^t_{u_0,\omega}$ which are unbounded (measured by the surrogate $u$), which means that the $\zeta,C$ in the G{\aa}rding's inequality \eqref{ineq:GardingIntro} and the errors in Egorov's theorem would become time-dependent (or the regularization would need to be time-dependent).
These issues present the main difficulties in proving Theorem \ref{thm:AM}.
In Section \ref{sec:symbol}, the regularization procedure of the symbol is presented and the basic properties are verified.
In Section \ref{sec:mixing} the main arguments in Theorem \ref{thm:AM} are given, namely an $H^{-p/2} \to H^{-\frac{d+3}{2}}$-type mixing estimate and its use, together with Assumption \ref{ass:Lyapunov-Flow-Assumption}, to absorb the low frequency error terms coming from Egorov's theorem and G{\aa}rding's inequality to yield Theorem \ref{thm:AM}. 

\section{Applications} \label{sec:Concrete}

Here we outline several relevant applications of our general framework to examples of interest, most notably stochastic flow for SDEs and flows generated by the stochastic Navier-Stokes equations. Additionally, we state some important applications to the advection diffusion equation with a stochastic source term and the existence of a unique stationary measure in the zero-diffusivity limit.

\subsection{Examples} \label{sec:examples}

\subsubsection{IID diffeomorphisms and stochastic flows}
A simple, but wide, class of examples our theorem applies to are iid random diffeomorphisms, such as the case of the stochastic flow of diffeomorphisms associated to an SDE on a compact, Riemannian manifold $(\mathcal{M},g)$.
Consider smooth divergence-free vector fields $X_0,X_1,...,X_r$ on $(\mathcal{M},g)$, then the SDE
\begin{align*}
\dee x_t = X_0(x_t) \dee t + \sum_{j=1}^r X_j(x_t) \circ \dee W_t,
\end{align*}
defines a stochastic flow of diffeomorphisms $x_t = \phi^t_\omega(x_0)$ (see \cite{Kunita1997-oa} for details).
As the vector fields are divergence free, $\phi^t_\omega$ is volume-preserving almost-surely, i.e. the Riemannian volume measure is almost-surely invariant.  
The projective process solves a similar SDE on the sphere bundle $\mathbb S \mathcal{M}$ \cite{baxendale1989lyapunov} (i.e. the unit tangent bundle), denoted 
\begin{align*}
\dee z_t = \tilde{X}_0(z_t) \dee t + \sum_{j=1}^r \tilde{X}_j(z_t) \circ \dee W_t, 
\end{align*}
with $z = (x,v)$ and the `lifted' vector fields satisfy 
\begin{align*}
\tilde{X} = \begin{pmatrix} V(x) \\ (I - v\otimes v) \grad X(x) v \end{pmatrix}. 
\end{align*}
As we explain in more detail below in Section \ref{sec:Checking}, \cite{dolgopyat2004sample} provides checkable sufficient conditions for these flows to satisfy Assumption \ref{ass:SpecGap}--\ref{ass:2pt}, at least when combined with the construction of $\psi_p$ found in \cite{bedrossian2022almost}; see also the earlier work of \cite{baxendale1988large,carverhill1987furstenberg,baxendale1989lyapunov}.
In particular, a quenched mixing estimate \eqref{ineq:quenched} was proved in \cite{dolgopyat2004sample}, specifically for any $\alpha \in (0,1)$, $\exists \gamma > 0$ (deterministic)
\begin{align}
\norm{f \circ \phi^n_\omega}_{C^{0,-\alpha}} \lesssim D(\omega) e^{-\gamma n}\norm{f}_{C^{0,\alpha}}, \label{ineq:C0quenched}
\end{align}
with $\EE D^2 < \infty$. 
Our results now additionally prove that for $0 < p \ll 1$ (the lack of $V$ implies we can take $\eps = 0$), 
\begin{align*}
\EE \norm{f \circ \phi^n}_{H^{-p}}^2 \lesssim e^{-\mu n}\norm{f}_{H^{-p}}^2. 
\end{align*}
Another concrete case of iid random diffeomorphisms are the Pierrehumbert flows \cite{blumenthal2023exponential}, namely the case of transport by alternating shear flows on $\mathbb T^d$. 
For example, in 2D, the time-one map random map $\phi_\omega:\T^d \to \T^d$ could be given by $\phi_\omega( (x,y) ) = (x_\ast,y_\ast)$, where
\begin{align*}
x_\ast & =  x + A(\omega)\sin (y + \gamma(\omega)) \\
y_\ast & =  y + A'(\omega)\sin (x_\ast + \gamma'(\omega)).
\end{align*}
where $A,A',\gamma,\gamma'$ are suitable independent random variables, for example each drawn uniformly from $(-\pi,\pi)$ suffices. 
Assumptions \ref{ass:SpecGap}--\ref{ass:2pt} were proved in \cite{blumenthal2023exponential}, wherein a quenched exponential mixing estimate such as \eqref{ineq:C0quenched} was proved and a suitable $\psi_p$ was constructed. Theorem \ref{thm:AM} now provides also the estimate \eqref{eq:annealedeq}. 

\subsubsection{Stochastic Navier-Stokes equations}
Let us briefly explain how to apply Theorem \ref{thm:AM} to the setting of stochastic Navier-Stokes in $\mathbb T^2$, as in works of \cite{bedrossian2022almost,bedrossian2021almost} and \cite{cooperman2024exponential}.
The PDE in question is given by the following in $\mathbb T^2$
\begin{equation}\label{eq:SNS}
\begin{dcases}
\,\dee u_t + \left( u_t \cdot \grad u_t + \grad p_t - \nu \Delta u_t\right) \dee t + Q \dee W_t \\ 
\,\Div u_t = 0 \, ;
\end{dcases}
\end{equation}
in 3D the viscosity must be replaced by a suitable hyperviscosity $-\nu \Delta \mapsto \nu (-\Delta)^2$ but this case otherwise also works. 
Let us now explain the operator $Q W_t$. 
Following the convention used in \cite{weinan2001ergodicity}, we define the following real Fourier basis for functions on $\T^d$ by
 \[\label{eq:Fourier-Basis}
 e_k(x) = \begin{cases}
 \sin(k\cdot x), \quad& k \in \Z^d_+\\
 \cos(k\cdot x),\quad& k\in \Z^d_-,
 \end{cases}
 \]
where $\Z_+^d = \{(k_1,k_2,\ldots k_d)\in \Z^d : k_d >0\}\cup\{(k_1,k_2,\ldots k_d)\in \Z^d \,:\, k_1>0, k_d=0\}$ and $\Z_-^d = - \Z_+^d$. 
We set $\Z^d_0 := \Z^d \setminus \set{0,\ldots, 0}$ and define $\{\gamma_k\}_{k\in \Z^d_0}$ a collection of full rank $d\times (d-1)$ matrices satisfying $\gamma^\top_k k = 0$, $\gamma_k^\top\gamma_k=\Id$, and $\gamma_{-k} = - \gamma_k$. Note that in dimension $d=2$, $\gamma_k$ is just a vector in $\R^2$ and is therefore given by $\gamma_k = \pm k^{\perp}/|k|$. In dimension three, the matrix $\gamma_k$ defines a pair of orthogonal vectors $\gamma_k^{1},\gamma_k^2$ that span the space perpendicular to $k$. 
Next, we define the natural Hilbert space on velocity fields $u:\T^d \to \R^d$ by
\begin{align}
  \Lbf^2 := \set{u \in L^2(\mathbb T^d;\Real^d)\,: \,\int u\, \dx = 0,\quad \Div u = 0},   
\end{align}
with the natural $L^2$ inner product. Let $W_t$ be a cylindrical Wiener process on $\Lbf^2$ with respect to an associated canonical stochastic basis $(\Omega,\mathscr{F},(\mathscr{F}_t),\P)$ and $Q$ a Hilbert-Schmidt operator on $\Lbf^2$, diagonalizable with respect the Fourier basis on $\Lbf^2$.
In the works \cite{bedrossian2021almost,bedrossian2022almost}, the operator $Q$ was assumed to satisfy the following regularity and non-degeneracy assumption
\begin{assumption}[Assumptions in \cite{bedrossian2021almost,bedrossian2022almost}] \label{a:Highs}
  There exists $\alpha$ satisfying $\alpha > \frac{5d}{2}$ and a constant $C$ such that
\[
	\frac{1}{C}\|(-\Delta)^{-\alpha/2}u\|_{\Lbf^2} \leq \|Qu\|_{\Lbf^2} \leq C\|(-\Delta)^{-\alpha/2}u\|_{\Lbf^2}.
\]
\end{assumption}
In the work \cite{cooperman2024exponential} the lower bound assumption was dropped and replaced with the following much weaker assumption. 
\begin{assumption}[(Essentially) Assumptions in \cite{cooperman2024exponential}] \label{ass:CR}
The $Q$ is compactly supported in frequency and satisfies the H\"ormander hypoellipticity condition given in \cite{HM06,HM08} in 2D and \cite{RomitoXu11} in 3D. 
We will additionally assume that all of the modes with $\abs{k}_{\ell^\infty} \leq 1$ are forced.
\end{assumption}
Theorem \ref{thm:AM} can be applied to both of these cases. 
We define our primary phase space of interest to be velocity fields with sufficient Sobolev regularity (under Assumption \ref{ass:CR} we can choose $\sigma > \frac{d}{2} + 3$ arbitrarily): 
\[
	\Hbf := \set{u \in H^\sigma(\T^d, \R^d)\, :\, \int u \, \dx = 0,\quad \Div u = 0}, \quad \text{where}\quad \,\sigma \in (\alpha-2(d-1), \alpha  - \tfrac{d}{2}). 
\]
Note we have chosen $\alpha$ sufficiently large to ensure that $\sigma > \frac{d}{2} + 3$ so that $\Hbf \hookrightarrow C^3$. 
Since we will need to take advantage of the ``energy estimates'' produced by the vorticity structure of the Navier-Stokes equations in $2D$, we find it notationally convenient to define the following dimension dependent norm
\begin{equation}\label{eq:vorticity-norm}
\|u\|_{\Wbf} := \begin{cases}
\|\curl u\|_{\Lbf^2} & d=2\\
\|u\|_{\Lbf^2} & d=3.
\end{cases}
\end{equation}

The following well-posedness theorem is classical (see e.g. \cite{KS}). 
\begin{proposition} \label{prop:WP}
For all initial data $u\in \Hbf$,
there exists a $\P$-a.s. unique, global-in-time, $\mathscr{F}_t$-adapted mild solution $(u_t)$ to \eqref{eq:SNS} satisfying $u_0 = u$. 
Moreover, $(u_t)$ defines a Feller Markov process on $\Hbf$ and the corresponding Markov semigroup has a unique stationary probability measure $\mu$ on $\Hbf$.
\end{proposition}

One then defines the Lagrangian flow map $\phi^t_{u,\omega} : \T^d \to \T^d$ as the solution map to the ODE
\begin{equation}\label{eq:flow-eq}
\partial_t\phi^t_{u,\omega} = u_t\circ\phi^t_{u,\omega}, \quad \phi^0_{u,\omega} = \Id. 
\end{equation}

\subsection{Checking the conditions for stochastic Navier-Stokes} \label{sec:Checking}

In this section we sketch some ideas required to verify the conditions of Theorem \ref{thm:AM} in the setting of stochastic Navier-Stokes. In the setting of stochastic flows, checking the conditions follow along similar lines and typically much easier to verify (see e.g \cite{Gess2021-wh}). 

Assumption \ref{ass:Lyapunov-Flow-Assumption} will be show below in Section \ref{sec:Assumption1-verif}.
Assumption \ref{ass:2pt} was proved for Assumption \ref{a:Highs} in \cite{bedrossian2022almost} and for Assumption \ref{ass:CR} in \cite{cooperman2024exponential}. 
Assumption \ref{ass:SpecGap} (i) -- (iii) regarding the construction of $\psi_p$ was proved directly in \cite{bedrossian2022almost} for the case of Assumption \ref{a:Highs}.
In \ref{sec:psi_p_construction} below, we briefly explain how to obtain Assumption \ref{ass:SpecGap} (i) -- (iii) for the degenerate noise Assumption \ref{ass:CR}. 

\subsubsection{Checking Assumption \ref{ass:Lyapunov-Flow-Assumption}}\label{sec:Assumption1-verif}

We must check Assumption \ref{ass:Lyapunov-Flow-Assumption} for the Lyapunov function $V_{\beta,\eta}$ given by
\begin{align*}
  V_{\beta,\eta}(u) = \langle\norm{u}_{\Hbf}\rangle^{2\beta} e^{\eta \eta_*\norm{u}_{\Wbf}^2},
\end{align*}
where $\eta_*$ is a 
This will be proved by making use of the following super Lypaunov bound proved in \cite{bedrossian2022almost} (the proof works for both Assumption \ref{a:Highs} and \ref{ass:CR}).
\begin{lemma}[ \cite{bedrossian2022almost} Lemma 3.7] \label{lem:TwistBd} Let $(u_t)$ be a solution to \ref{eq:SNS}. There exists a $\gamma_* >0$, such that for all $0\leq \gamma < \gamma_*$, $T>0$, $r\in (0,3)$, $\kappa \geq 0$, and $V(u) = V_{\beta,\eta}$ where $\beta \geq 1$ and $0 < e^{\gamma T}\eta < 1$, there exists a constant $C = C(\gamma,T,r,\kappa,\beta,\eta) >0$ such that the following estimate holds
  \begin{equation}\label{eq:Twistbd}
  \EE_u \exp\left(\kappa \int_0^T\norm{u_s}_{\Hbf^r}\ds\right)\sup_{0\leq t\leq T}V^{e^{\gamma t}}(u_t) \leq C V(u).
  \end{equation}
\end{lemma}

Assumption \ref{ass:Lyapunov-Flow-Assumption} reduces to proving the following:
\begin{lemma}
Let $(u_t)$ be a solution to \eqref{eq:SNS}. For all $b \geq 0$, $\eta \in (0,1/a_*)$, $\beta \geq 2$ and $k \leq k_0$, the following estimate holds,
\[
\E_{u,x}\left[\left(|D^k\phi^1_{\omega,u}| + |D^k(\phi^1_{\omega,u})^{-1})|\right)^b V_{\beta,\eta}(u_1)\right]^{a_*} \leqc_{b,k,a,\beta,\eta} V_{\beta + b(k-1),\eta}(u).
\]
\end{lemma}
\begin{proof}

To apply Lemma \ref{lem:TwistBd} we must first bound derivatives of $\phi$ in terms of something that can be controlled
\begin{claim} For each $k\in \N$, there exists a $c_k$ satisfying $c_k = 1$ such that the following estimate holds
\[
\sup_{s\in[0,t]}|D^k\phi^s_{\omega,u}| \leqc_k \sup_{s\in[0,1]}\langle\|u_s\|_{C^{k}}\rangle^{k-1}\exp\left(c_k\int_0^1\|Du_s\|_{L^\infty}\dee s\right)
\]
\end{claim}
\begin{proof}
The proof will be by induction on $k$. Clearly it is true for $k = 1$. We assume it is true for all $j\leq (k-1)$. Using the equation \eqref{eq:flow-eq} we can estimate $|D^k\phi^t_{\omega,u}|$ using the Fa\`{a} di Bruno formula 
\[
\partial_t|D^k\phi^t_{\omega,u}|\leqc_{k} \sum_{n = 1}^k|D^{n}u_s\circ\phi^t_{\omega,u}| \sum_{C_{n,k}} \prod_{j=1}^k|D^j\phi^t_{\omega,u}|^{m_j},
\]
where the sum is over the set
\[
\mathcal{C}_{n,k} = \{(m_1,\ldots,m_k) \in \Z^k_{\geq 0}\,:\, 1\cdot m_1 +2\cdot m_2 +\ldots +k\cdot m_k = k, \quad m_1+\ldots +m_k = n\}.
\]
Separating out the leading order term in the outer sum $n=1$, $|Du\circ\phi^t_{\omega,u}||D^k\phi^t_{\omega,u}|$, and applying Gr\"onwall to the remaining terms, we obtain the following estimate
\[
  \sup_{s\in[0,1]}|D^k\phi^s_{\omega,u}|\leqc_{k} \sup_{s\in[0,1]}\left(\sum_{n = 2}^{k}\|D^{n}u_s\|_{L^\infty}\sum_{C_{n,k}} \prod_{j=1}^{k-1}|D^j\phi^s_{\omega,u}|^{m_j}\right) \exp\left(\int_0^1\|Du_s\|_{L^\infty}\dee s\right),
\]
where we used that in $m_k = 0$ if $n\geq 2$ and hence the product is only up to $k-1$. Using the induction hypothesis we have for $n\geq 2$ and $(m_1,\ldots, m_k) \in C_{n,k}$ that
\[
\begin{aligned}
  \prod_{j=1}^{k-1}|D^j\phi^t_{\omega,u}|^{m_j} &\leqc_{k} \sup_{s\in[0,1]}\langle\|u_s\|_{C^{k-1}}\rangle^{\sum_{j=2}^{k-1}(j-1)m_j}\exp\left(\sum_{j=2}^{k-1}c_jm_j\int_0^1\|Du_s\|_{L^\infty}\dee s\right),\\
  &\leqc_{k} \sup_{s\in[0,1]}\langle\|u_s\|_{C^{k-1}}\rangle^{k-2}\exp\left((c_k -1)\int_0^1\|Du_s\|_{L^\infty}\dee s\right),
\end{aligned}
\]
where in the second inequality we used that $\sum_{j=2}^{k-1}(j-1)m_j = k-n \leq k-2$ and define $c_k := 1+ \sup_n \sup_{(m_1,\ldots,m_k)\in C_{n,k}}\sum_{j=2}^{k-1}c_jm_j$. Substituting this back into the expression for $|D^k\phi^t_{\omega,u}|$ we obtain the desired estimate.
\end{proof}

Likewise we obtain a similar estimate for the inverse:
\begin{claim} For each $k\in \N$, there exists a $\tilde{c}_k>0$ satisfying $\tilde{c}_k = 1$ such that the following estimate holds
\[
\sup_{s\in[0,t]}|D^k(\phi^s_{\omega,u})^{-1}| \leqc_k \sup_{s\in[0,1]}\langle\|u_s\|_{C^{k}}\rangle^{k-1}\exp\left(\tilde{c}_k\int_0^1\|Du_s\|_{L^\infty}\dee s\right)
\]
\end{claim}
\begin{proof}
To prove this, we will make use of the following estimate on derivatives of the inverse of a diffeomorphism which can be deduced from the Fa\`{a} di Bruno formula,  
\[
|D^k(\phi^1_{\omega,u})^{-1}| \leqc_k \sum_{\mathcal{B}_{k}}\prod_{j=1}^k\|D^j\phi^1_{\omega,u}\|^{m_j}_{L^\infty},
\]
where the sum is over the set
\[
\mathcal{B}_k = \{(m_1,\ldots,m_k)\in \Z^k_{\geq 0}\,:\, 1\cdot m_1 + 2\cdot m_2 + \ldots + k\cdot m_k = 2k-2, \quad m_1 + \ldots + m_k = k-1\}.
\]
Substituting in the estimate for $|D^k\phi^1_{\omega,u}|$ we obtain
\[
  |D^k(\phi^1_{\omega,u})^{-1}| \leqc_k \left(\sum_{\mathcal{B}_{k}} \sup_{s\in[0,1]}\langle\|u_s\|_{C^{k}}\rangle^{\sum_{j=1}^k (j-1)m_j}\right)\exp\left(\tilde{c}_k\int_0^1\|Du_s\|_{L^\infty}\dee s\right),
\]
for some suitable $\tilde{c}_k > 0$. Using that $\sum_{j=1}^k(j-1)m_j = 2k-2 - (k-1) = k-1$, we obtain the desired estimate.

\end{proof}

Using these estimates we can now prove the desired result by applying the super-Lyapunov bound \eqref{eq:Twistbd} to the above estimates assuming that the regularity of $u$ is sufficiently high $\sigma \geq k_0$. 

\end{proof}

\subsubsection{Construction of $\psi_p$ with degenerate noise} \label{sec:psi_p_construction}

Likely the simplest way to construct $\psi_p$ is through a spectral perturbation method applied to $\hat{P}^p$, which is the method employed in e.g. \cite{bedrossian2021almost,bedrossian2022almost,blumenthal2023exponential}.
The first step is to prove that $\mathbf{1}$ is the unique, dominant eigenvector for $\hat{P}$ in $C^1_V$ (the corresponding eigenvalue is of course $1$), which amounts to verifying a spectral gap for $\hat{P}^\ast$ in the dual Lipschitz metric (i.e. the Wasserstein-1 norm). This is typically done with a weak Harris' theorem; the proof of this geometric ergodicity found in \cite{bedrossian2022almost}, which closely follows \cite{HM06}, can be used for both Assumption \ref{a:Highs} and Assumption \ref{ass:CR}.
Then the observation that $\lim_{p \to 0} \hat{P}^p = \hat{P}$ in the strong operator topology implies the existence of a $\psi_p$ satisfying Assumption \ref{ass:SpecGap} (i) and (iii) through classical spectral perturbation theory. 
Verifying (ii) is a standard convexity `trick' which proves that  $\Lambda(p) = \lambda p + o(p) $ as $p \to 0$, where $\lambda$ is the Lyapunov exponent. See \cite{bedrossian2022almost,blumenthal2023exponential} for expositions of this argument.

\subsection{Passive scalars}
As a final application of Theorem \ref{thm:AM} we address specifically the advection-diffusion equation on $\T^d$
\begin{align}
  \partial_t f_t + u_t \cdot \nabla f_t = \kappa \Delta f_t, \quad f|_{t=0} = f_0.   \label{eq:PassScalIntroAD}
\end{align}
The transfer operator for this equation can be written as follows
\begin{align}
f_t = \EE_{\tilde{W}} f_0 \circ (\phi^t_{\omega,u_0})^{-1}, \label{eq:ftAdvection}
\end{align}
where $\phi^t_{\omega,u_0}$ solves the SDE
\begin{align*}
\dee \phi^t_{\omega,u_0}  = u_t ( \phi^t_{\omega,u_0}  ) + \sqrt{2 \kappa} \dee \tilde{W}_t, \quad \phi^0_{\omega,u_0}(x) = x.
\end{align*}
In \eqref{eq:ftAdvection}, the notation $\EE_{\tilde{W}}$ refers to the expectation with respect only to the Brownian motions $\tilde{W}_t$ (which are of course independent from $(u_t)$).
For the case of Pierrehumbert \cite{blumenthal2023exponential} and stochastic Navier-Stokes under Assumption \ref{a:Highs} \cite{bedrossian2021almost}.
it was proved in those references that Assumptions \ref{ass:Lyapunov-Flow-Assumption}--\ref{ass:2pt} all hold \emph{uniformly} in $\kappa \in [0,\kappa_0]$ for some small $\kappa_0$ (note that this is far from obvious).
As the proof of Theorem \ref{thm:AM} is quantitative as well, in these cases, Theorem \ref{thm:AM} also holds uniformly in $\kappa \in [0,\kappa_0]$ for solutions to \eqref{eq:PassScalIntroAD}.
It seems plausible that one could carry this out also for Assumption \ref{ass:CR} however as of writing, this result has not yet appeared in the literature. 

There is one last notable consequence of Theorem \ref{thm:AM} which pertains to the limiting system arising in Batchelor-regime passive scalar turbulence \cite{bedrossian2022batchelor}, specifically the system
\begin{subequations}
\begin{align}
& \dee u_t + (u_t \cdot \grad u_t + \grad p_t - \nu \Delta u_t)\dee t = Q \dee W_t \label{eq:NSEPS} \\
& \Div u_t = 0   \\ 
& \dee f_t + (u_t \cdot \nabla f_t - \kappa \Delta f_t) \dee t = b \dee \xi_t. \label{eq:ADPS} \\
\end{align}
\end{subequations}
For all $\kappa > 0$, $\exists !$ stationary measure for the joint $(u_t,f_t)$ process, which we denote $\mu^{\kappa}$. 
In \cite{bedrossian2022batchelor} it was proved that any sequence $\set{\mu^{\kappa_n}}_{n=0}^\infty$ such that $\kappa_n \to 0$ has a subsequence which converges weak-$\ast$ to a measure supported on $\Hbf \times H^{-\delta}$ for all $\delta > 0$ which is a stationary measure for the $\kappa = 0$ system. 
One can use Theorem \ref{thm:AM} to prove that in fact, there is a unique stationary measure for the $\kappa=0$ system supported on $\Hbf \times H^{-\delta}$ (and hence a unique limit for  all such convergent subsequences as $\kappa \to 0$).
\begin{theorem}
Under Assumption \ref{a:Highs} or Assumption \ref{ass:CR}, there exists a unique stationary measure to the $\kappa = 0$ system \eqref{eq:NSEPS}--\eqref{eq:ADPS} supported on $\Hbf \times H^{-\delta}$. 
\end{theorem}

\begin{proof}
The proof is based on Birkhoff's ergodic theorem and the contraction in $H^{-\delta}$. Suppose that there exist two measure $\mu^1$ and $\mu^2$ for the process $(u_t,f_t)$ on $\Hbf \times H^{-\delta}$ which are stationary for the $\kappa = 0$ system. Using ergodic decomposition we may assume that they are ergodic. Hence, by Birkhoff's ergodic theorem, there exist two sets $\hat{A}_1, \hat{A}_2$ such that for $i,j \in\{1,2\}$, $\mu^i(A_j) = \delta_{ij}$ and for every bounded Lipschitz $\varphi$ on $\Hbf \times H^{-\delta}$, and every initial data $(u_0,f_0) \in \hat{A}_i$ we have
\[
\lim_{T\to \infty}\frac{1}{T}\sum_{t=0}^{T-1} \E\varphi(u_t,f_t) \to \int \varphi \dee \mu^i
\]
Moreover since $\mu^1, \mu^2$ both project down to a unique stationary measure $\mu$ for the Navier-Stokes system on $\Hbf$, it must be that the projections $A_1,A_2 \subset \Hbf$ of $\hat{A}_1,\hat{A}_2$ are both full $\mu$ measure and hence have a non-empty intersection $A_1\cap A_2$ which is also full $\mu$ measure. It follows that we may choose initial data $(u_0,f^1_0)\in \hat{A}_1$ and $(u_0,f^2_0)\in \hat{A}_2$, such that $u_0\in A_1\cap A_2$ and we can write
\begin{equation}\label{eq:Diff-mu}
\left|\int \varphi \dee \mu^1 - \int \varphi \dee \mu^2\right| \leq \lim_{T\to \infty} \frac{1}{T}\sum_{t=0}^{T-1}\E|\varphi(u_t,f^1_t) - \varphi(u_t,f^2_t)|\leqc \lim_{T\to \infty} \frac{1}{T}\sum_{t=0}^{T-1}\E\|f^1_t - f^2_t\|_{H^{-\delta}}.
\end{equation}
Note that since the velocity $u_t$ is the same for $f^1_t$ $f^2_t$, the difference $\tilde{f}_t = f^1_t - f^2_t$ solves the advection-diffusion equation
\[
\partial_t \tilde{f}_t + u_t \cdot \nabla \tilde{f}_t = 0,
\]
and hence by Theorem \ref{thm:AM} the right-hand side of \eqref{eq:Diff-mu} goes to zero as $T\to \infty$, implying that $\mu^1 = \mu^2$.
\end{proof}

%

\section{Constructing the Symbol}

\label{sec:symbol}

In this section, we use the function $\psi_p$ from Assumption \ref{ass:SpecGap} to construct a pseudo-differential operator $a_{p,\eps} \in S^{-p}_{1-\eps}$ which exhibits exponential decay under the map $T^*P$, up to a lower order remainder. This involves a certain regularization scheme of the finite regularity symbol $a_p (u,x,\xi)= \psi_p(u,x,\xi/|\xi|)|\xi|^{-p}$, where we mollify $\psi_p$ at the scale $h>0$, where $h \sim |\xi|^{-\eps}$ as $|\xi|\to \infty$. This scheme is reminiscent of para-differential calculus \cite{metivier2008differential}, although it differs in the details.

\subsection{Lower bound on $\psi_p$}

First we show that assumption that $\psi_p$ is lower bounded on bounded sets actually implies a $1/V(u)$ lowerbound on $\psi_p$ for all $u \in \mathbf H$.

\begin{lemma}
For all Lyapunov functions $V(u) = V_{\beta,\eta}(u)$ with $\beta \geq  1$ and $\eta \in (0,1/a_*)$ and $\forall \abs{p} \ll 1$, $\psi_p$ satisfies the following lower bound 
\[
  V_{\beta,\eta}(u)^{-1} \leqc_\eta |\psi_p(u,x,v)|. 
\]
\end{lemma}

\begin{proof}
Let $r>0$ and define $B_r = \{V(u) \leq r\}$ for $V(u) = V_{\beta,\eta}(u)$ with $\beta,\eta$ as above. Now let $\tau = \inf\{n>0\:\, u_n\in B_R\}$. Our first step is to show that for each $\gamma >0$, there exists an $r>0$ big enough so that
\begin{equation}\label{eq:exit-time-est}
  \P(\tau >n) \leqc_\gamma V(u)e^{-\gamma n}.
\end{equation}
Such a bound follows relatively easily by standard techniques in Markov chains, which we outline now for completeness. Indeed, the discrete Dynkin formula implies that
\[
\E e^{\gamma \tau}V(u_\tau) = V(u_0) + \E\sum_{k=1}^\tau e^{\gamma k} \left(PV(u_{k-1}) - e^{-\gamma} V(u_{k-1})\right).
\]
By the super Lyapunov property $PV \leq e^{-\gamma}V + K_\gamma$ mentioned in Remark \ref{eq:super-lyapunov-remark} and the definition of $\tau$ this implies the following exponential estimate on $\tau$
\[
\left(r-\frac{K_\gamma}{e^{\gamma}-1}\right)\E e^{\gamma \tau} \leq V(u_0).
\]
Hence, taking $r>\frac{K_\gamma}{e^{\gamma}-1}$ and using Markov's inequality gives \eqref{eq:exit-time-est}.

Next, we note that by the spectral gap condition (Assumption \ref{ass:SpecGap}) on $\hat{P}^p$ we have that $\psi_p$ is given by
\begin{equation}\label{eq:limit-formula}
\psi_p = \lim_{n\to\infty} e^{\Lambda(p)}\hat{P}^p\bm{1} \quad\text{in}\quad \mathrm{Lip}_V. 
\end{equation}
Indeed, $e^{\Lambda(p)}\hat{P}^p$ has a dominant simple eigenvalue $1$ with a spectral gap. Let $\pi_p$ be the rank one spectral projector associated with $e^{-\Lambda(p)}$ and recall that $\psi_p = \pi_p\bm{1}$.
Then by the spectral gap assumption and Gelfand's formula $\lim_{n\to \infty} e^{\Lambda(p)n}(\hat{P}^p)^n(I-\pi_p)\bm{1} =0$ in $\mathrm{Lip}_V$. It follows that
\[
  \lim_{n\to \infty} e^{\Lambda(p)n}(\hat{P}^p)^n\bm{1} = \lim_{n\to \infty} e^{\Lambda(p)n}(\hat{P}^p)^n\pi_p\bm{1} = \psi_p,\quad \text{in } \mathrm{Lip}_V
\]

The limit formula \eqref{eq:limit-formula} implies that for each $u\in B_r$, we have
\[
  \psi_p(u,x,v) = \lim_{n\to\infty} e^{\Lambda(p)n}\E|D_x\phi^n_{\omega,u} v|^{-p},
\]
which is uniform over $u$ in $B_r$. By Assumption \ref{ass:SpecGap} there exists a $c_r>0$ such that
\[
  \inf_{u\in B_r} |\psi_p(u,x,v)| \geq c_r.
\]
With this in mind, let $N_0 > 0$ be large enough so that
\[
G_n(u,x,v) := e^{\Lambda(p)n} \E |D_x \phi^n_{\omega, u} v |^{-p} \geq c_r
\]
for all $n \geq N_0$ and $u\in B_r$, $(x, v) \in P \T^d$. Now we can write $G_n(u,x,v)$ as
\[
\begin{aligned}\label{eq:G-Markov-identity}
G_n(u,x,v) &= \E e^{\Lambda(p)\tau} |D_x \phi^\tau_{\omega, u} v|^{-p} \cdot e^{\Lambda(p)(n - \tau)} |D_{x_\tau} \phi^{n - \tau}_{\theta^\tau \omega, u_\tau} v_{\tau}|^{-p} \\
& = \E \left[ e^{\Lambda(p)\tau} |D_x \phi^{\tau}_{\omega, u} v|^{-p} \cdot
\E \left( e^{\Lambda(p)(n - {\tau})} |D_{x_{\tau}} \phi^{n- \tau}_{\theta^{\tau} \omega, u_{\tau}} v_{{\tau}}|^{-p} \bigg| \mathcal{F}_{\tau} \right) \right] \,\\
&= \E\left[e^{\Lambda(p)\tau} |D_x \phi^{\tau}_{\omega, u} v|^{-p} \cdot
G_{n-\tau}(u_{\tau},x_\tau, v_\tau)\right]
\end{aligned}
\]
Here, $(\mathcal{F}_n)$ denotes the standard filtration and $\mathcal{F}_{\tau}$ is the corresponding stopped sigma-algebra, consisting of measurable sets $K$ for which $K \cap \{ \tau \leq n \} \in \mathcal{F}_n$ for all $n\geq 0$. In the last line of the above inequality, we used the strong Markov property to conclude
\[
  G_{n-\tau}(u_\tau ,x_\tau,v_\tau) = \E \left( e^{\Lambda(p)(n - {\tau})} |D_{x_{\tau}} \phi^{n - \tau}_{\theta^{\tau} \omega, u_{\tau}} v_{{\tau}}|^{-p} \bigg| \mathcal{F}_{\tau} \right)
\]

We now bound $\psi_p$ for $u\in B^c_r$ as follows. Since for $n \geq N_0$ and $u \in B_r$, $G_n(u,x,v) \geq c_r$. Therefore, by \eqref{eq:G-Markov-identity}, for any $u\in B_r$
\begin{equation}
\begin{aligned}
  G_n(u,x,v) &\geq \E\left[{\bf 1}_{N_0 \leq n-\tau} e^{\Lambda(p)\tau} |D_x \phi^{\tau}_{\omega, u} v|^p \cdot
G_{n-\tau}(u_{\tau},x_\tau, v_\tau)\right]\\
&\geq c_{r}\E\left({\bf 1}_{N_0 \leq n-\tau} e^{-\Lambda(p)\tau}|D_x\phi^\tau_{\omega,u}v|^{-p}\right).
\end{aligned}
\end{equation}
Taking $n \to \infty$ and using that $\tau < \infty$ almost-surely, we conclude by monotone convergence that
\[
\psi_p = \lim_{n\to\infty}G_n(u,x,v) \geq c_r \E (e^{- \Lambda(p) \tau} |D_x \phi^{\tau}_{\omega, u}v|^{-p}) \geq c_r \E(e^{\Lambda(p)\tau}|D_x\phi^\tau_{\omega,u}v|^p)^{-1} \, .
\]

Next, we note that upon choosing $\gamma$ large enough (and consequently $r$ large enough), we have that
\[
\begin{aligned}
  \E(e^{\Lambda(p)\tau}|D_x\phi^\tau_{\omega,u}v|^p) &\leq \sum_{k > 0}\E(e^{2\Lambda(p)k}|D_x\phi^k_{\omega,u}v|^{2p})^{1/2}P(\tau \geq k)^{1/2}\\
    &\leqc_\gamma V(u) \sum_{k > 0}C^k e^{(2\Lambda(p) - \gamma)k/2} \leqc_\gamma V(u)
\end{aligned},
\]
where we used the Assumption \ref{ass:Lyapunov-Flow-Assumption} with $k=1$ in the last line to bound using the Markov property
\[
  \E |D_x\phi^k_{\omega,u}v|^{2p} \leq C \E\left[|D_x\phi^{k-1}_{\omega,u}v|^{2p}V(u_{k-1})\right] \leqc C^k V(u),
\]
for some $C\geq 1$.
\end{proof}

\subsection{Regularization scheme}

Now we describe a regularization scheme that gives rise to a symbol $a_{p,\eps}$ using properties of $\psi_p$. The main result of this section, Lemma \ref{lemma:apSymbolBddness}, shows that this symbol belongs to the symbol class $S^{-p}_{1-\eps}$, and is an approximate eigenfunction of $T^*P$ (see appendix \ref{appendix:PDOs} for a definition of the symbol class).

For each $h \in (0,1)$, we take $\{\psi_{p}^h\}_{h\in(0,1)}$ to be a family of a $C^\infty$ mollifications of $\psi_p$ in the variables $(x,v) \in S^* M$ up to scale $h$. 
We impose the following requirements on such a mollification. First, we require that for all $u \in \mathbf H$, we have
\begin{align}
\min_{(x,v) \in S^*M} \psi_p(u,x,v) & \leq \min_{(x,v) \in S^*M} \psi_p^h(u,x,v) \\
 \max_{(x,v) \in S^*M} \psi_p^h(u,x,v) &\leq \max_{(x,v) \in S^*M}\psi_p(u,x,v).
\end{align}
We emphasize that $C$ is independent of $u$.
Letting $d_{S^*M} : S^*M\times S^*M \to [0,\infty)$ be the geodesic distance between any two points on $S^*M$, and let $\mathrm{Lip}(S^*M)$ be the corresponding Lipschitz norm: 
 \begin{align}
\|\psi_p^h(u)\|_{\mathrm{Lip}(S^*M)}&\leq C  \|\psi_p(u)\|_{\mathrm{Lip}(S^*M)},\\
\|\psi_p^h(u) - \psi_p(u)\|_{L^\infty(S^*M)} & \leq Ch \|\psi_p(u)\|_{\mathrm{Lip}(S^*M)}.
\end{align}
Regarding higher order derivatives, given any parametrization $\nu : W \to S^*M$ where $W \subset \R^{2n-1}$ is an open set, we require that
\begin{align}
\|\psi_p^h(u,\nu(z))\|_{C^k(W)} &\leq C_{k,\nu} h^{1-k} \|\psi_p(u)\|_{\mathrm{Lip}(S^*M)}, \text{ for all } k > 1.
\end{align}
The existence of the family $\{\psi_p^h\}$ is standard. For instance, when $ M =\T^d$, we can define $\psi_p^h$ using convolution by a standard mollifier.
 One can generalize such a mollification scheme to arbitrary $ M$ by using the exponential map or a partition of unity.

Using the boundedness of $\psi_p$ w.r.t. the $\mathrm{Lip}_V$ norm in \eqref{eq:lipnorm}, the above conditions imply that for each  $\eta \in (0,1/a^*)$ and $\beta \geq 1$, 
\begin{align}\label{eq:psi_p_delta_to_V}
\|\psi_p^h(u)\|_{L^\infty(S^*M)} +\|\psi_p^h(u)\|_{\mathrm{Lip}(S^*M)} + \frac{1}{h} \|\psi_p^h(u) - \psi_p(u)\|_{L^\infty(S^*M)}\lesssim_{\beta,\eta} V_{\beta,\eta}(u)
\end{align}
and
\begin{align}
 \|\psi_p^h(u,\nu(z))\|_{C^k(W)} &\leq C_{k,\nu,\beta,\eta}h^{1-k} V_{\beta,\eta}(u)
\end{align}
for any $k > 1$ and $\nu: W \to S^*M$ described above. Moreover,
\begin{align}\label{eq:psi_p_lower}
\frac{1}{C_{\beta,\eta}V_{\beta,\eta}(u)}  \leq \psi_p^h(u,x,\xi)
\end{align}
independently of $h > 0$.

Fix a smooth, non-negative dyadic partition of unity $\{\chi_N\}_{N=2^k,\ k \in \mathbb N_0}$ of $\R$, such that for each $N \geq 2$, we have $\mathrm{spt} \ \chi_N \subseteq B_{2N}\setminus B_{N/2}$,  $\chi_1 \subseteq B_1$, and
\[
\sum_{N} \chi_N(z) \equiv 1
\]
for all $z \in \R$.
For convenience, given any $(x,\xi) \in T^* M$, we define 
\[
\hat \xi := \frac{\xi}{|\xi|} \in S_x^* M
\]
For longer expressions, we shall use the notation $ (\xi)^{\wedge}=\hat \xi $.
Let $\eps \in (0,\frac{1}{2})$. We define
\begin{align}\label{eq:apeps_definition}
a_{p,\eps}(u,x,\xi) := \sum_{N \geq 2} \chi_N(|\xi|) \psi^{N^{-\eps}}_p\left(u,x,\hat \xi\right) \frac{1}{|\xi|^p}.
\end{align}
For each $\eps>0$, this defines an approximation of $a_p(u,x,\xi)$, which improves as $|\xi| \to \infty$. The parameter $\eps$ determines the rate of convergence. Since the role of $a_p$ is that of an  eigenfunction to the operator $T^*P$, we think of $a_{p,\eps}$ as an approximate eigenfunction with error
\[
r_{p,\eps} = T^*P a_{p,\eps} - e^{-\Lambda(p)}a_{p,\eps}.
\]
The following Lemma shows that $a_{p,\eps}$ and $r_{p,\eps}$ belong to appropriate symbol classes from which we can construct pseudo-differential operators, with $r_{p,\eps}$ having strictly lower order than $a_{p,\eps}$ (see Appendix \ref{appendix:PDOs}):

\begin{lemma} \label{lemma:apSymbolBddness}  Let $\eta \in (0,1)$ and $\eps \in (0,\frac{1}{4})$.
For each $u \in \mathbf H$,  $a_{p,\eps}(u) \in S^{-p}_{1-\eps}(T^* M)$, and  $r_{p,\eps}(u) \in S^{-p-\eps}_{1-2\eps}(T^* M)$. 
For each $k \in \N_0$, the seminorms of $a_{p,\eps}$ (defined in Appendix \ref{sec:quantization_on_manifolds})
are bounded as follows: for any $\beta \geq 1$,  
\begin{align}
[a_{p,\eps}(u)]^{-p,1-\eps}_{k} \lesssim_{k,p,\eps,\beta,\eta} V_{\beta,\eta}(u)\label{eq:apepsBddness}.
\end{align}

On the other hand, there exists $ \beta(k,p,\eps) \geq 1 $  (depending only on $k$, $p$ and $\eps$) such that
\begin{align}
[r_{p,\eps}(u)]_{k}^{-p-\eps,1-2\eps} \lesssim_{k,p,\eps,\eta} V_{\beta({k,p,\eps}),\eta}(u). \label{eq:apepsApproximation}
\end{align}

\end{lemma}

\begin{proof} We prove the bounds above in multiple steps:\\

\noindent{\em Step 1:} We prove \eqref{eq:apepsBddness}.
Fix a coordinate chart $  \varkappa_\iota :U'_\iota \to U_\iota$ from the atlas $\{\psi_\iota\}_{\iota \in J}$ as defined in Appendix \ref{appendix:PDOs} and set $\zeta_{\iota} = \varkappa_{\iota}^{-1}$.
It suffices to show that for any such $\iota \in J$, and any two multi-indices $\alpha,\alpha'$, we have
\begin{align}
|\partial_{x}^\alpha \partial_\xi^{\alpha'} \{a_{p,\eps}(u,\zeta_\iota(x),D_x\zeta_{\iota}^{-\top}\xi)\}| \lesssim_{\iota,\alpha,\alpha',\beta,\eta,p,\eps} V_{\beta,\eta}(u) \langle \xi\rangle^{-p + \eps|\alpha|-(1-\eps)|\alpha'|}.
\end{align}
In proving this, we will omit the dependence of implicit constants on the parameters listed above. We will also drop the subscript on the coordinate chart and write $\zeta = \zeta_{\iota}$.
To show this, first observe that
\[
\psi^{N^{-\eps}}_p\left(u,\zeta(x),(D_x\zeta^{-\top} \xi)^\wedge\right) \frac{1}{|\xi|^p}
\]
is a  $-p$-homogeneous function in $\xi$. Moreover, given any cone  $\mathcal C \subsetneq \R^n$, we can parameterize $\xi \neq 0$ in spherical coordinates using
\[
\xi(\lambda,y) = \lambda \sigma(y)
\]
where $\lambda >0$ and $\sigma(z)$ maps an open subset of $\R^{n-1}$ to $\S^{d-1}$. Then, with this parametrization, we have
\[
\psi^{N^{-\eps}}_p\left(u,\zeta(x),(D_x\zeta^{-\top} \xi)^\wedge\right) \frac{1}{|\xi|^p} =\psi^{N^{-\eps}}_p\left(u,\zeta(x),(D_x\zeta^{-\top} \sigma(y))^\wedge\right) \frac{1}{\lambda ^p |D_x\zeta^{-\top} \sigma(z)|^p}
\]
We call $z = (x,y)$ and
\[
\nu(z) = (\zeta(x),( D_x\zeta^{-\top}\sigma(y))^\wedge) \in S^*M,
\]
so that $\nu$ defines a smooth parametrization of $S^*M$ in some open subset.
Furthermore, the partial derivatives $\partial_\xi$ transform as follows under the change of coordinates:
\[
\frac{\partial}{\partial \xi_i} =  \sigma_i(y)\frac{\partial}{\partial \lambda} + \frac{1}{\lambda}\sum_{j =1}^{n-1} S_{ij}(y) \frac{\partial}{\partial y_j}
\] 
where $S(y)$ is a $n\times (n-1)$ matrix given by the Moore-Penrose pseudo-inverse of $D_y \sigma^\top$, i.e.
\[
\sum_{k=1}^{n-1} S_{ik}(y) \frac{\partial \sigma_j}{\partial y_k} = \delta_{ij} -\sigma_i\sigma_j,
\]
and for any $y$, the range of $S(y)$ (as a matrix) is $\{\sigma(y)\}^\perp$.

Thus, for all $\xi \in \R^d$ such that $\frac{\xi}{|\xi|}$ is in the range of $\sigma$, we use the $-p$ homogeneity in $\lambda$ to bound
\begin{align}
\left|\partial_{x}^\alpha \partial_\xi^{\alpha'} \left( \psi^{N^{-\eps}}_p\left(u,\zeta(x),(D_x\zeta^{-\top}\xi)^\wedge\right) \frac{1}{|\xi|^p}\right)\right| &\lesssim \frac{\|\psi_p^{N^{-\eps}}(u,\nu(z))\|_{C^{\alpha + \alpha'}_z}}{\lambda^{-p + |\alpha'|}} \\
&\lesssim V_{\beta,\eta}(u) \frac{N^{\eps(\alpha + \alpha')}}{\lambda^{-p + |\alpha'|}} \\
&\lesssim V_{\beta,\eta}(u) \frac{N^{\eps(\alpha + \alpha')}}{|D_x\zeta^{-\top} \xi|^{-p + |\alpha'|}}.
\end{align}
By covering $\xi \in \R^n$ by finitely many cones for we which can construct such a smooth map $\sigma$, we recover the bound for all possible $\xi \in \R^n$.
In particular, when $|D_x\zeta^{-\top} \xi| \sim N$ (as is the case when $|D_x\zeta^{-\top} \xi|$ is in the support of $\chi_N$)
 the above is bounded by
\[
\lesssim V_{\alpha',\eta}(u) |D_x\zeta^{-\top}\xi|^{-p- (1-\eps)|\alpha'| + \eps |\alpha|}.
\]
On the other hand, using the parametrization $\sigma(y) = \frac{\xi}{|\xi|}$ once more, we have
\[
\chi_N(|D_x\zeta^{-\top}\xi|) = \chi\left(\frac{\lambda}{N}|D_x\zeta^{-\top}\sigma(y)|\right).
\]
Then, it is straightforward to show that for all $\lambda \sim N$ for which the above expression is nonzero that
\[
|\partial_{x}^\alpha \partial_\xi^{\alpha'} (\chi_N(|D_x\zeta^{-\top}\xi|))| \lesssim \frac{1}{\lambda^{|\alpha'|}} \lesssim \frac{1}{|D_x\zeta^{-\top}\xi|^{|\alpha'|}}.
\]
We then apply the multivariable chain rule to each term  in the sum which defines $a_{p,\eps}$ as in \eqref{eq:apeps_definition}
to recover \eqref{eq:apepsBddness}.\\

\noindent {\em Step 2a:} We now show \eqref{eq:apepsApproximation} for $k=0$. It suffices to show that for all $\beta \geq 1$, $\eta \in (0,1)$ and $(x,\xi)\in T^*M$, we have
\[\label{eq:apepsapproximation0}
|r_{p,\eps}(u,x,\xi)| \lesssim_{\beta,\eta,p,\eps} V_{\beta,\eta}(u) \langle  \xi\rangle^{-p-\eps},
\]
and then take $\beta = \beta(k,p,\eps)$ as defined in Step 2b below. Once again, dependence of constants on the parameters enumerated above is implied.
 Using the identity $T^*P a_{p} = e^{-\Lambda(p)}a_p$, we  write
\begin{align}
T^*P a_{p,\eps}(u) - e^{-\Lambda(p)} a_{p,\eps}(u) &= T^*P [a_{p,\eps}(u) -  a_{p}(u)]  - e^{-\Lambda(p)} (a_{p,\eps}(u ) - a_{p}(u)).\end{align}
The difference appearing above can be expressed as follows:
\[
a_{p,\eps}(u,x,\xi) - a_{p}(u,x,\xi) = -\chi_1(|\xi|)\psi_p(u,x,\hat\xi)\frac{1}{|\xi|^p} +  \sum_{N \geq 2} \chi_N(|\xi|) [\psi^{N^{-\eps}}_p -\psi_p](u,x,\hat \xi)  \frac{1}{|\xi|^p}.
\]
Since the first term is compactly supported, it is trivial to bound it by $\frac{V(u)}{|\xi|^{p+\eps}}$. As for the second term, we have
\begin{align}
\left|   \sum_{N \geq 2} \chi_N(|\xi|) [\psi^{N^{-\eps}}_p -\psi_p](u,x,\hat \xi)  \frac{1}{|\xi|^p}\right| &\lesssim  \sum_{N \geq 2} \chi_N(|\xi|) \|\psi^{N^{-\eps}}_p(u) -\psi_p(u)\|_{L^\infty(S^*M)} \frac{1}{|\xi|^p}\\
&\lesssim \sum_{N \geq 2} \chi_N(|\xi|)  \frac{ V_{\beta,\eta}(u) N^{-\eps}}{|\xi|^p}\\
&\lesssim \frac{V_{\beta,\eta}(u)}{|\xi|^{p+\eps}}.
\end{align}
Thus, 
\[
|a_{p,\eps}(u,x,\xi) - a_{p}(u,x,\xi)| \lesssim \frac{V_{\beta,\eta}(u)}{|\xi|^{p+\eps}}.
\]
We can apply this bound to the same difference under $T^*P$, in conjunction with Assumption \ref{ass:Lyapunov-Flow-Assumption}, 
\begin{align}
|T^*P[a_{p,\eps}(u,x,\xi) - a_{p}(u,x,\xi)]| &= \E|a_{p,\eps}(u_1,\phi^1(x),(D_x\phi^1)^{-\top}\xi) - a_{p}(u,\phi^1(x),(D_x\phi^1)^{-\top}\xi))| \\
& \lesssim \E\left[\frac{V_{\beta,\eta}(u_1)}{| (D_x\phi^1)^{-\top}\xi|^{p+\eps}}\right] \lesssim \frac{V_{\beta,\eta}(u)}{|\xi|^{p+\eps}}.
\end{align}
In summary,
\begin{align}
|r_{p,\eps}(u,x,\xi)| \lesssim \frac{V_{\beta,\eta}(u)}{|\xi|^{p+\eps}}.
\end{align} 
On the other hand,  we have the uniform bound $|r_{p,\eps}(u,x,\xi)| \lesssim V_{\beta,\eta}(u)$,
since $a_{p,\eps}$ is zero near the pole of $\frac{1}{|\xi|^p}$. Taking the minimum of these two bounds, we conclude \eqref{eq:apepsapproximation0}.\\

\noindent {\em Step 2b:} As in step 1, we fix a coordinate chart $\zeta^{-1} = \zeta_{\iota}^{-1}$, and suppress the dependence of constants on the parameters. For higher order derivatives of $r_{p,\eps}$, 
we let $|\alpha| +|\beta| = k \geq 1$, and use the crude bound
\begin{align}
|\partial_{x}^\alpha \partial_\xi^\beta \{r_{p,\eps}(u,\zeta(x),D_x\zeta^{-\top}\xi)\}| &\leq  
|\partial_{x}^\alpha \partial_\xi^\beta \{a_{p,\eps}(u,\zeta(x),D_x\zeta^{-\top}\xi)\}|\\
&\quad + \E[|\partial_{x}^\alpha \partial_\xi^\beta \{a_{p,\eps}(u_1, \phi^1\circ \zeta(x),(D_{\zeta(x)}\phi^1)^{-\top} D_x \zeta^{-\top}\xi)\}| ].
\end{align}
Then, we recycle \eqref{eq:apepsBddness} to bound the above by
\begin{align}
|\partial_{x}^\alpha \partial_\xi^\beta \{r_{p,\eps}(u,\zeta(x),D_x\zeta^{-\top}\xi)\}| &\lesssim V_{\beta,\eta}(u)\langle \xi\rangle^{-p+\eps |\alpha| -(1-\eps)|\beta|} \\
&\quad +\E[Q_{k}(\phi^1)^k V_{\beta,\eta}(u_1)\langle( D_{\zeta(x)}\phi^{1})^{-\top}\xi\rangle^{-p+\eps |\alpha| - (1-\eps)|\beta|} ].
\end{align}
Now, in the latter term, we have
\[
\langle( D_{\zeta(x)}\phi^{1})^{-\top}\xi\rangle^{-p+\eps |\alpha| - (1-\eps)|\beta|} \leq \mathcal Q_1(\phi^1)^{p + \eps |\alpha| +(1-\eps)|\beta|} \langle\xi\rangle^{-p+\eps |\alpha| - (1-\eps)|\beta|}
\]
Hence, applying Assumption \ref{ass:Lyapunov-Flow-Assumption},
\begin{align}
&|\partial_{x}^\alpha \partial_\xi^\beta \{r_{p,\eps}(u,\zeta(x),D_x\zeta^{-\top}\xi)\}|\\
&\quad \lesssim  \left(V_{\beta,\eta}(u) +\E[ \mathcal Q_k(\phi^1)^{k+p + \eps |\alpha| +(1-\eps)|\beta|}V_{\beta,\eta}(u_1)]\right)\langle \xi\rangle^{-p+\eps |\alpha| - (1-\eps)|\beta|}  \\
&\quad \lesssim V_{\beta,\eta}(u)\langle \xi\rangle^{-p+\eps |\alpha| -(1-\eps)|\beta|}.
\end{align}
Now, observe that since $k\geq 1$,
\[
-p+\eps |\alpha| - (1-\eps)|\beta| \leq -p-\eps+2\eps |\alpha| - (1-2\eps)|\beta|.
\]
Combining this with Step 2a, we conclude \eqref{eq:apepsApproximation} for all $k\geq 0$.
\end{proof}

As an application of G\r arding's inequality (Proposition \ref{prop:ManifoldGarding}), the  lemma above implies that the ``norm'' $f\mapsto \langle\mathrm{Op}(a_{p,\eps})f,f\rangle$ defines the same topology as $H^{-\frac{p}{2}}(M)$, provided one has control on lower frequencies.

\begin{corollary}[Quantitative G\r arding's inequality] \label{cor:GardingSpecialized} Under the hypotheses of Lemma \ref{lemma:apSymbolBddness}, we have
\begin{align}
\frac{1}{C_{p,\eps,\beta,\eta} V_{\beta,\eta}(u)} \|f\|_{H^{-\frac{p}{2}}}^2 - C_{p,\eps,\beta,\eta} V_{\beta,\eta}(u) \|f\|_{H^{-\frac{d+3}{2}}}^2 \leq \langle \mathrm{Op}(a_{p,\eps}(u)) f,f\rangle \leq C_{p,\eps,\beta,\eta} V_{\beta,\eta}(u) \|f\|_{H^{-\frac{p}{2}}}^2.
\end{align}

\begin{proof}
Observe that by \eqref{eq:psi_p_lower} and the definition of $a_{p,\eps}$ in equation \eqref{eq:apeps_definition}, we have 
\begin{align}\label{eq:apepsEllipticity}
  \frac{1}{C_{p,\eps,\beta,\eta}V_{\beta,\eta}(u)| \xi|^{p}}  \leq a_{p,\eps}(u,x,\xi).
\end{align} 
for all $|\xi| \geq 1$, and for any $\beta \geq 1$. Thus, 
 we combine G\r arding's inequality in Proposition \ref{prop:ManifoldGarding} in the case of $m = -p$ and $\frac{m}{2} -s= -\frac{d+3}{2}$ with the seminorm bounds \eqref{eq:apepsBddness} and \eqref{eq:apepsEllipticity}. 
\end{proof}
\end{corollary}

\section{Low regularity mixing}

\label{sec:mixing}

In this section, we use the results of Section \ref{sec:symbol} to show prove Theorem \ref{thm:AM}, i.e. that $f_n$ decays exponentially fast in $H^{-\frac{p}{2}}$.  We break this up into two steps: first, we exhibit time 1 exponential decay of high frequencies of $f_1$ through the pseudo-differential operator $\mathrm{Op}(a_{p,\eps})$ as in Lemma \ref{lemma:time1decay}. Second, we show that low frequencies decay exponentially for long times, as in Lemma \ref{lemma:lowFreqMixing}. The latter of these two results does not rely on  $a_{p,\eps}$. Rather, it is a consequence of the two-point geometric ergodicity in Assumption \ref{ass:2pt}. We combine these two lemmas to prove Theorem \ref{thm:AM} in Section \ref{sec:proof}.

\subsection{Short time, high frequency decay}

Lemma \ref{lemma:apSymbolBddness} and Egorov's Theorem (Theorem \ref{theorem:egorov}) imply that under conjugation by the time-1 flow map $\phi^1$, the operator $\mathrm{Op}(a_{p,\eps})$ decays by a factor of $e^{-\Lambda(p)}$, plus a lower order remainder:
 
\begin{lemma}\label{lemma:time1decay}
Let $\eps \in (0,\frac{1}{4})$, $p\in (0,1)$, and $\eta \in (0,1/a^*)$.  Then there exists some $\beta^\prime = \beta^\prime(p,\eps) \geq 1$ (depending only on $p$ and $\eps$) such that
\begin{equation}
|\E[\langle \mathrm{Op}(a_{p,\eps}(u_1))f_1, f_1\rangle] - e^{-\Lambda(p)} \langle \mathrm{Op}( a_{p,\eps}(u)) f, f\rangle | \lesssim_{p,\eps,\eta} V_{\beta^\prime,\eta}(u) \|f\|_{H^{-\frac{p+\eps}{2}}}^2.
\end{equation}
\end{lemma}

\begin{proof}
By Egorov's theorem (Theorem \ref{theorem:egorov}), Lemma \ref{lemma:apSymbolBddness} and Assumption \ref{ass:Lyapunov-Flow-Assumption},
\begin{align}
&|\E[\langle \mathrm{Op}(a_{p,\eps}(u_1))f_1, f_1\rangle] - \langle \mathrm{Op}(T^*P a_{p,\eps}(u)) f, f\rangle | \\
&\quad \lesssim  \E\Big[\mathcal Q_{k_0}(\phi)^{k_0}V_{0,\eta}(u_1) \Big ] \|f\|_{H^{-\frac{p+1-2\eps}{2}}}^2 \lesssim V_{\beta_1(k_0),\eta}(u),
\end{align}
for some $\beta_1(k_0) \geq 1$. On the other hand, by the boundedness of pseudo-differential operators (Proposition \ref{prop:manifoldPDObddness}) and Lemma \ref{lemma:apSymbolBddness} above,
\[
\left| \langle \mathrm{Op}(T^*P a_{p,\eps}(u)) f, f\rangle  - e^{-\Lambda(p)}\langle \mathrm{Op}(a_{p,\eps}(u))f,f\rangle\right| \lesssim_\eps V_{\beta_2(k_0,p,\ep),\eta}(u) \|f\|_{H^{-\frac{p+\eps}{2}}}^2,
\]
for some $\beta_2(k_0,p,\ep)$. Combining these two bounds, and fact that $\eps < 1-2\eps$, we finish the proof choosing $\beta^\prime$ to be the bigger of $\beta_1(k_0)$ and $\beta_2(k_0,p,\ep)$.
\end{proof}

\subsection{Long time, low frequency decay} \label{sec:lowDecay}
Next, we show how the assumption of 2-point mixing in Assumption \ref{ass:2pt} implies the exponential decay of $f_n$ in a very low-regularity norm.

\begin{lemma}\label{lemma:lowFreqMixing} There are absolute constants  $\alpha,p\in (0,1)$ such that the following holds. For any $\beta \geq 1$, $\eta \in (0,1/a^*)$  and any $f_0 \in H^{-p/2}$ such that $\int_{ M} f_0 \,\dx = 0$, we have 
\begin{align}\label{eq:lowFreqDecay}
\E[V_{\beta,\eta}(u_n)\|f_n\|_{H^{-\frac{d+3}{2}}}^2] \lesssim  { e^{-\alpha n}}V_{\beta,\eta}(u_0)\|f_0\|_{H^{-\frac{p}{2}}}^2.
\end{align}
Here, we recall that $f_n = T^n f_0 = f_0\circ (\phi^n_{u_0,\omega})^{-1}$.
\end{lemma}

\begin{proof}  We break the proof into two steps. For convenience, we  write with $V(u) = V_{\beta,\eta}(u)$. \\

\noindent \textit{Step 1:}
First, we show that there exists $\alpha_0 >0$ such that 
\begin{align}
 \E[V(u_n)\|f_n\|_{H^{-\frac{d+3}{2}}}^2] \lesssim e^{-\alpha_0 n}V(u)\|f_0\|_{L^2}^2, \label{eq:L2toLowFreq}
\end{align}
for all $n \in \N_0$.
To prove the bound above, 
let $K(x,y)$  be the kernel of the operator 
\[
\mathrm{Op}(\langle \xi\rangle^{-\frac{d+3}{2}})^* \mathrm{Op}(\langle \xi\rangle^{-\frac{d+3}{2}}) \in S^{-d-3}_1(T^*M).
\]
  In particular,
$K\in C^{3-\eps}(\mathcal M\times \mathcal M)$ for all $\eps \in(0,1)$. We define
\[
K_\times (x,y) =K(x,y) - \iint_{M\times M} K(x',y') \dx' \dy',
\]
so that $K_\times$ is mean zero.
 Then, since $f_0$ is mean zero,
\begin{align}
 \E[V(u_n)\|f_n\|_{H^{-\frac{d+3}{2}}}^2] & = \iint_{ M\times  M}  \E[V(u_n)K_{\times}(\phi^{n}(x),\phi^{n}(y)) ]f_0(x) f_0(y) 
 \dx \dy\\
 &\leq \|\E[ V(u_n)K_\times(\phi^{n}(x),\phi^{n}(y))  ]\|_{L^2_{x,y}(M^2)}\|f_0\|_{L^2(M)}^2.
\end{align}
Then, estimate \eqref{eq:L2toLowFreq} follows from the two-point mixing bound of Assumption \ref{ass:2pt}: for all $q>0$ sufficiently small, and all $(x,y) \in  M\times  M$, 
\[
|\E[ V(u_n)K_\times(\phi^{n}(x),\phi^{n}(y))]| \lesssim_q V(u_0)d_{ M} (x,y)^{-q} e^{-\alpha_0 n} \| K\|_{L^\infty} 
\]
for all $n \in \N_0$, and $d_M(x,y)$ denotes the distance between $x$ and $y$,
Thus, by taking $q < \frac{d}{2}$, we have
\[
\|\E[ V(u_n)K_\times(\phi^{n}(x),\phi^{n}(y)) ] \|_{L^2}^2 \lesssim e^{-\alpha_0 n}.
\]

\noindent \textit{Step 2:} We show that there exists $\alpha_1 >0$ such that
\begin{align}
\E[V(u_n)\|f_n\|_{H^{-\frac{d+3}{2}}}^2] \lesssim e^{\alpha_1 n}V(u_0)\| f_0\|_{H^{-1}}^2. \label{eq:H-1toLowFreq}
\end{align}
To prove this, 
we use the cocycle property of $D\phi^n$ and Assumption  \ref{ass:Lyapunov-Flow-Assumption}  in the case $k = 1$, and take $\alpha_1>0$ large enough so that 
\begin{align}\| \E [V(u_n)|D \phi^n|^2]\|_{L^\infty} & \leq e^{\alpha_1} \E[V(u_{n-1})|D \phi^{n-1}|^2] \\
&\leq e^{2\alpha_1}\E[V(u_{n-2})|D \phi^{n-2}|^2] \\
&\leq \ldots\\
& \leq e^{\alpha_1 (n-1)} \E[V(u_{1})|D \phi^{1}|^2] \leq e^{\alpha_1 n}V(u_0). 
\end{align}
Then, by writing $f_0 = \mathrm{div} ( \nabla \Delta^{-1}f_0)$, where $\Delta$ is the Laplace-Beltrami operator and $\mathrm{div}$ denotes the divergence on $M$, we can integrate by parts to get
\begin{align}
\E[V(u_n)\|f_n\|_{H^{-\frac{d+3}{2}}}^2] &\lesssim \|\nabla^2 K\|_{L^\infty}\| \E [V(u_n)|D_x \phi^n|^2]]\|_{L^\infty}  \|\nabla \Delta^{-1} f_0\|_{L^1}^2 \lesssim e^{\alpha_1 n} \|\nabla \Delta^{-1} f_0\|_{L^2}^2.
\end{align}
Finally, we note
\[
\|\nabla \Delta^{-1} f_0\|_{L^2}^2 = \langle (-\Delta)^{-1} f_0,f_0\rangle \lesssim \|f_0\|_{H^{-1}}^2,
\]
from which \eqref{eq:H-1toLowFreq} follows. \\

\noindent \textit{Step 3:} In the final step, we interpolate the two estimates to get  \eqref{eq:lowFreqDecay}. 
For each $h >0$, let  $f^{(h)}_0$ be a mollification of $f_0$ that satisfies
\[
\|f^{(h)}_0\|_{L^2} \lesssim h^{-\frac{p}{2}} \|f_0\|_{H^{-\frac{p}{2}}}, \quad \|f_0-f^{(h)}_0\|_{H^{-1}} \lesssim h^{1-\frac{p}{2}}  \|f_0\|_{H^{-\frac{p}{2}}}.
\]
Then, combining \eqref{eq:L2toLowFreq} and \eqref{eq:H-1toLowFreq}, we have
\begin{align}
\E[V(u_n)\|f_n\|_{H^{-\frac{d+3}{2}}}^2] &\lesssim \E[V(u_n)\|S_n f^{(h)}_0\|_{H^{-\frac{d+3}{2}}}^2]  + \E[V(u_n)\|S_n (f_0-f^{(h)}_0)\|_{H^{-\frac{d+3}{2}}}^2] \\
&\lesssim V(u_0) (e^{-\alpha_0 n}h^{-p} + e^{\alpha_1 n} h^{2-p})  \|f_0\|_{H^{-\frac{p}{2}}}^2.
\end{align}
We now choose $h = e^{-\frac{\alpha_0 n}{2 p}}$, and
\[
0< p \leq \frac{\alpha_0}{2(\alpha_0 + \alpha_1)}
\]
so that
\[
e^{-\alpha_0 n}h^{-p} \leq e^{-\alpha_0n/2}, \quad \text{ and }  \quad  e^{\alpha_1 n} h^{2-p} = e^{(\frac{\alpha_0}{2}+\alpha_1 -\frac{\alpha_0}{2p} )n} \leq e^{-\alpha_0n/2}.
\]
Taking $\alpha = \frac{\alpha_0}{2}$, this gives us  \eqref{eq:lowFreqDecay}.

\end{proof}

\subsection{Proof of the main theorem}\label{sec:proof}

Now we are ready to prove the main Theorem \ref{thm:AM}

\begin{proof}[Proof of Theorem \ref{thm:AM}] We fix the constants $\alpha, p \in (0,1)$ as in Lemma \ref{lemma:lowFreqMixing} and set $\eps = \frac{1}{5}$. In the statement of the main theorem, we shall prove the estimate with $\delta = \frac{p}{2}$. 

By Lemma \ref{lemma:time1decay}, there exists some $\beta^\prime$ such that for each $\eta^\prime \in (0,1/a^*)$
\begin{align}\label{eq:time1decayUpperBound}
\E[\langle \mathrm{Op}(a_{p,\eps}(u_1)) f_1,f_1\rangle] \leq  e^{-\Lambda(p)}\langle \mathrm{Op}(a_{p,\eps}(u))f_0,f_0\rangle+ C_{\eta^\prime}  V_{\beta^\prime,\eta^\prime}(u) \|f_0\|_{H^{-\frac{p+\eps}{2}}}^2.
\end{align}
Using interpolation, there exists $\varrho >0$ such that for any $h >0$, there is some constant $C_h$ for which the following holds:
\[
V_{\beta^\prime,\eta^\prime}(u) \|f_0\|_{H^{-\frac{p+\eps}{2}}}^2 \leq h\|f_0\|_{H^{-\frac{p}{2}}}^2 + C_h V_{\beta^\prime,\eta^\prime}(u)^{\varrho} \|f_0\|_{H^{-\frac{d+3}{2}}}^2.
\]
Applying G\r arding's inequality, Corollary \ref{cor:GardingSpecialized}, to the first term on the right-hand side, we have
\begin{equation}\label{eq:Gardinginterpolation}
V_{\beta^\prime,\eta^\prime}(u) \|f_0\|_{H^{-\frac{p+\eps}{2}}}^2 \leq h \langle \mathrm{Op}(a_{p,\eps}(u))f_0,f_0\rangle +C_{h} V_{\beta^\prime,\eta^\prime}(u)^\varrho \|f_0\|_{H^{-\frac{d+3}{2}}}^2.  
\end{equation}
We now set
\[
\beta_* := \varrho \beta^\prime, \quad \eta_*:= \varrho \eta^\prime,
\] 
so that 
$V_{\beta^\prime,\eta^\prime}(u)^\varrho = V_{\beta_*,\eta_*}(u)$. Under this re-scaling, $\beta_*\geq 1$ is an absolute constant, while $\eta_* \in (0,1/(a^* \varrho))$. For brevity sake, since $\beta_*$ and $\eta_*$ are now fixed, we will drop the subscripts and write $V =V_{\beta_*,\eta_*}$.

Combined with \eqref{eq:time1decayUpperBound}, this yields
\begin{align}\label{eq:time1decayUpperBoundRedux}
\E[\langle \mathrm{Op}(a_{p,\eps}(u_1)) f_1,f_1\rangle] \leq  (e^{-\Lambda(p)} + h)\langle \mathrm{Op}(a_{p,\eps}(u))f_0,f_0\rangle+ C_{h}  V(u) \|f_0\|_{H^{-\frac{d+3}{2}}}^2.
\end{align}
We now take $h = e^{-\Lambda(p)/2}-e^{-\Lambda(p)} $ and set $\mu = \min\{\Lambda(p)/2,\alpha\}$.
Then, by the Markov property, 
\begin{align}\label{eq:time1decayUpperBoundRedux2}
\E[\langle \mathrm{Op}(a_{p,\eps}(u_{n+1})) f_{n+1},f_{n+1}\rangle|\mathcal F_n] \leq  e^{-\mu}\langle \mathrm{Op}(a_{p,\eps}(u_n))f_n,f_n\rangle+ C V(u_n) \|f_n\|_{H^{-\frac{d+3}{2}}}^2.
\end{align}
By taking the full expectation of both sides, and iterating the above inequality, we have
\[
\E[\langle \mathrm{Op}(a_{p,\eps}(u_{n})) f_{n},f_{n}\rangle ] \leq e^{-\mu n} \langle \mathrm{Op}(a_{p,\eps}(u))f_0,f_0\rangle + C \sum_{k =0}^{n-1} e^{-(n-1-k)\mu} \E[ V(u_k) \|f_k\|_{H^{-\frac{d+3}{2}}}^2].
\] 
Then, by applying Lemma \ref{lemma:lowFreqMixing} to the sum (recalling that $\mu \leq \alpha$) to get
\begin{align}
\E[\langle \mathrm{Op}(a_{p,\eps}(u_{n})) f_{n},f_{n}\rangle ] &\leq e^{-\mu n} \langle \mathrm{Op}(a_{p,\eps}(u))f_0,f_0\rangle + C_\eta V(u) e^{-\mu n}\|f_0\|_{H^{-\frac{p}{2}}}^2\\
&\lesssim_\eta  V(u) e^{-\mu n} \|f_0\|_{H^{-\frac{p}{2}}}^2.
\end{align}
We apply Corollary \ref{cor:GardingSpecialized} once more to deduce 
\[
\E\left[\frac{1}{V(u_n)}\|f_n\|_{H^{-\frac{p}{2}}}^2\right] \lesssim_\eta  V(u) e^{-\mu n} \|f_0\|_{H^{-\frac{p}{2}}}^2 + \E[V(u_n) \|f_n\|_{H^{-\frac{d+3}{2}}}^2].
\]
for each $n$. Applying Lemma \ref{lemma:lowFreqMixing} once more to the second term on the right-hand side completes the proof.
\end{proof}

\appendix
\label{sec:appendix}
\section{Pseudo-differential calculus}\label{appendix:PDOs}
Here, we review the requisite results on pseudo-differential calculus on compact $d$-dimensional Riemannian manifolds $( M,g)$. 

In this section we fix a finite atlas $\{\varkappa_\iota : U_\iota \to U'_\iota\}_{\iota \in J}$, where each $\iota \in J$, we have that $U_\iota$ is an open set in $M$, and $V_\iota$ is an open set in $\R^n$. It will also be convenient to define $\zeta_\iota = \varkappa^{-1}_\iota$.

\subsection{Quantization on manifolds}
\label{sec:quantization_on_manifolds}

\begin{definition} Fix $m \in \R$ and $\rho \in (\frac{1}{2},1]$. We define the H\"ormander symbol classes $S_\rho^m(T^* M)$ to be the set of $a \in C^\infty(T^* M)$ satisfying the following. Given any trivialization $\Phi: T^* M|_U \to \R^{2n}$, we have
\[
\Big|\partial_x^{\alpha} \partial_{\xi}^{{\alpha'}} (a \circ \Phi^{-1})(x,\xi)\Big| \leq C_{\alpha,{\alpha'},\Phi} \langle \Phi^{-1}(x,\xi)\rangle^{m - \rho|{\alpha'}| + (1-\rho) |\alpha|},
\]
for all multi-indices $\alpha,{\alpha'}$, and all $(x,\xi)$ in the range of $\Phi$. In the above context, the  bracket $\langle z\rangle$ is given by $ \sqrt{1+ |z|^2}$ for any $z \in T^* M$, where
is defined in terms of the metric $|z|^2 = g_{\pi( z)}(z,z)$.

We define a family of seminorms that determine a Fr\' echet topology on $S_\rho^m$ as follows. For 
  each $k >0$, we define the seminorms
\begin{align}\label{eq:symbol_seminorms}
[a]_{k}^{m,\rho} = \sup_{\iota} \sup_{(x,\xi) \in U'_\iota \times \R^n} \sup_{|\alpha|,|{\alpha'}| \leq k}\Big|\partial_x^{\alpha} \partial_{\xi}^{\alpha'} (a(\zeta_\iota(x),D_x\zeta_\iota^{-\top}\xi))\langle D_x\zeta_\iota^{-\top}\xi\rangle^{-m + \rho|{\alpha'}| - (1-\rho) |\alpha|}  \Big|.
\end{align}

\end{definition}

Below, we give a generalization of the Kohn-Nirenberg quantization scheme for manifolds, following \cite{pflaum1998normal}.

\begin{definition}
Fix $\chi \in C^\infty(T M)$ to be a non-negative cutoff function satisfying the following: first, $\chi \equiv 1$ in an open neighborhood of $0_{T M}$; second, the map $(x,v) \mapsto (x,\exp_x(v))$ defines a diffeomorphism from the support of $\chi$ to an open neighborhood of the diagonal $\{(x,x) \ : \ x \in  M\} \subset  M\times  M$. 

Then, given any  symbol $a \in S^m_\rho$, we set
\[
\mathrm{Op}(a)f(x) = \frac{1}{(2\pi)^n} \iint_{T_x M\times T_x^*  M  } \chi(x,v) e^{i \langle v, \xi\rangle} a(x,\xi) f(\exp_x(v)) \,\dv \dee\xi 
\]
for any $f \in C^\infty( M)$.

\end{definition}

\begin{remark}
We remark that while this quantization depends on the choice of $\chi$, this construction is unique modulo the addition of infinitely smoothing operators. In contrast to \cite{faure2008semi}, we are only concerned with the principal symbol of our operator, although using this quantization scheme is convenient  in Lemma \ref{lemma:lowFreqMixing}.
An alternative, but equally viable quantization scheme can be found in Theorem 14.1 in \cite{zworski2022semiclassical}.
\end{remark}

\subsection{Function spaces on Riemannian manifolds}

Throughout this paper, we shall consider a general compact, smooth, $d$-dimensional Riemannian manifold $ M$. Here, we fix the function spaces throughout the paper. We define $L^2( M)$ to be the $L^2$ space on $ M$ with respect to the usual Riemannian volume form, with inner product
\[
\langle f,g\rangle = \int_{ M} f(x)g(x)\, \dx,
\]
for any two real valued $f,g \in L^2(M)$. Here, $\dx$ is shorthand for $\dee\mathrm{Vol}(x)$.

\begin{definition} 
We say $f \in H^s( M)$ to be the closure of $C^\infty( M;\C)$ with respect to the norm 
\[
\|f\|_{H^s}^2 :=\| f\|_{L^2}^2 +  \|\mathrm{Op}(\langle \xi\rangle^s) f\|_{L^2}^2.
\]
\end{definition}

We now define a norm-like quantity on the  group  of diffeomorphisms in $C^k( M, M)$, in order to track the regularity of such maps.

\begin{definition}\label{def:manifoldCkNorm}
   Let $k\geq 1$.  Given any $C^k$ diffeomorphism $\phi : M\to M$, we define the functional
\[
\mathcal Q_k(\phi) :=1+ \sup_{\iota,\zeta} \sup_{1\leq |\alpha |\leq k}  \max\{\|\partial_x^\alpha (\varkappa_\zeta\circ \phi\circ\varkappa_\iota^{-1})\|_{L^\infty(V_\iota \cap \phi^{-1}(U_\zeta);\R^n)},\|\partial_x^\alpha (\varkappa_\zeta\circ \phi^{-1}\circ\varkappa_\iota^{-1})\|_{L^\infty(V_\iota \cap \phi(U_\zeta);\R^n)}  \}.
\]
In the above, we set $\|\cdot \|_{L^\infty(\emptyset)} = 0$.
\end{definition}

\subsection{Estimates on pseudo-differential operators}

The first result we need is that the operators defined above are bounded between Sobolev spaces.  These results are classical when $ M$ is replaced with Euclidean space $\R^n$. 

\begin{proposition}\label{prop:manifoldPDObddness} Let $m \in \R$ and $\rho \in (\frac{1}{2},1]$.
Let $a \in S^m_\rho$. Then for each $s \in \R$, $\mathrm{Op}(a) $ is a bounded operator from $H^s$ to $H^{s-m}$. More precisely, for each $s \in \R$, there exists some $k \geq 0$ such that
\[
\|\mathrm{Op}(a) \|_{H^{s}\to H^{s-m}} \leq C_{m,\rho,s} [a]_{k}^{m,\rho}.
\]
In particular, for any $f \in H^{\frac{m}{2}}$,
\[
\langle \mathrm{Op}(a)f,f\rangle \leq C_{m,\rho} [a]_{k}^{m,\rho} \|f\|_{H^{\frac{m}{2}}}^2.
\]
\end{proposition}

We now state the weak G\r arding inequality on $ M$. 
\begin{proposition}\label{prop:ManifoldGarding} Let $m \in \R$ and $\rho \in (\frac{1}{2},1]$, and $c_0,c_1\in(0,\infty)$.
Let $a \in S^m_\rho (T^*  M)$, and suppose that it satisfies 
\[
c_0 |\xi|^m  \leq \mathrm{Re}(a(x,\xi))  
\]
for all $(x,\xi) \in T^* M$, with $|\xi| > c_1$. Then,  for all $f \in H^{\frac{m}{2}}( M)$, and $s >0$, there exists some $k$ such that
\[
 \frac{ c_0}{2}\|f\|_{H^{\frac{m}{2}}}^2 -  C_{m,s,\rho,c_0,c_1}\left\langle [a]^{m,\rho}_{k} \right\rangle  \|f\|_{H^{\frac{m}{2} - s}}^2 \leq \langle\mathrm{Op}(a)f,f\rangle_{L^2} .
\]

\end{proposition}

We now state a version of Egorov's theorem on manifolds.

\begin{theorem}\label{theorem:egorov} Let $m \in \R$ and $\rho \in (\frac{1}{2},1]$.
Let $a \in S^m_\rho(T^* M)$, $A = \mathrm{Op} (a)$,  and let $\phi:  M\to  M$ be a smooth diffeomorphism. Define the cotangent lift of $\phi$ to be the map  $\widetilde \phi : T^* M \to T^* M$ given by
\[
\widetilde \phi: (x,\xi) \mapsto (\phi(x), D_x\phi^{-\top} \xi).
\]
for any $x \in  M$ and $\xi \in T_x^* M$. Moreover, given any  map $\psi$ between manifolds, we define its pullback $\psi^* f = f\circ \psi$ for scalar valued $f$ where the composition is well-defined.

The operator $A^\phi := \phi^* A (\phi^{-1})^{*}$ can be decomposed
\[
A^\phi = \mathrm{Op}({\widetilde \phi}^* a) + R.
\]
Here, $\widetilde \phi^* a \in S^m_\rho$. Moreover, for all $s$, there exists $k_0  \in \N_0$ such that
\[
\|A^\phi\|_{H^{s} \to H^{s-m}}+ \|R\|_{H^{s}\to H^{s-m+2\rho-1}} \lesssim C_{m,s,\rho}\mathcal Q_{k_0}(\phi)^{k_0} [a]_{k_0}^{m,\rho}.
\]

\end{theorem}
\begin{remark} \label{remark:comments_on_proof}
While these results do not appear to be available in the literature as stated, their analogues on $\R^n$ are well-known, with the caveat that such results usually do not quantify the implicit constant in terms of seminorms on $a$. See  \cite{taylor1981pseudodifferential,hormander2007analysis}.
To recover the results on $M$, one can use a partition of unity and an atlas via a standard construction, and Egorov's theorem for properly supported pseudo-differential operators on $\R^n$ (see for instance \cite{hormander2007analysis},  Theorem 18.1.17).

 The quantitative bounds in terms of the seminorms on $a$, as well as  $\mathcal Q_k(\phi)$ in the case of Theorem \ref{theorem:egorov}, are direct consequences of the proofs of these results, as these seminorms appear in bounding error terms in the asymptotic expansions of pseudo-differential operators (see \cite{hormander2007analysis},  Theorems 18.1.7 and 18.1.8).
\end{remark}

\addcontentsline{toc}{section}{References}
\bibliography{refs}
\bibliographystyle{plain}

\end{document}